\documentclass[a4paper, 11pt]{amsart}

\usepackage[utf8]{inputenc}
\usepackage[all,cmtip]{xy}
\usepackage[english]{babel}
\usepackage{mathtools}
\usepackage{graphicx}
\usepackage{multicol}

\usepackage{amsmath,amssymb,amsthm,mathrsfs}
\usepackage[alphabetic]{amsrefs}
\let\amsrtimes=\rtimes

\usepackage{xfrac}
\usepackage{pifont}

\usepackage{array,tabularx}
\usepackage{xcolor}

\usepackage{tikz,tikz-cd}
\usetikzlibrary{positioning, shapes.geometric}
\usetikzlibrary{arrows,arrows.meta}
\usetikzlibrary{calc}
\tikzcdset{arrow style=tikz, diagrams={>={Computer Modern Rightarrow[width=5.5pt,length=3pt]}}}

\usepackage{enumitem}
\usepackage{cleveref}

\usepackage[retainorgcmds]{IEEEtrantools}

\theoremstyle{plain}
\newtheorem{thm}{Theorem}[section] 
\newtheorem*{thm*}{Theorem}
\newtheorem*{mainthm}{Main Theorem}
\newtheorem{prop}[thm]{Proposition}

\newtheorem{lem}[thm]{Lemma}

\newtheorem{hypothesis}[thm]{Hypothesis}

\counterwithin{equation}{section}

\theoremstyle{definition}
\newtheorem{defn}[thm]{Definition}

\newtheorem{rem}[thm]{Remark}
\newtheorem*{rem*}{Remark}

\newcommand*{\myproofname}{Proof of Theorem \ref{1thm:inertness}}

\newcommand*{\myproofnames}{Proof of Proposition \ref{prop:pdhmlgy}}

\newcommand{\Z}{\varmathbb{Z}}
\newcommand{\Q}{\varmathbb{Q}}

\newcommand{\C}{\varmathbb{C}}

\usepackage{geometry}
\usepackage{libertine}
\usepackage[libertine]{newtxmath}
\renewcommand{\mathbb}{\varmathbb}

\input xypic 
\input xy 
\xyoption{all} 
\usepackage{amssymb} 
\usepackage{bbm}

\newcommand{\conn}{\ensuremath{\#}}

\newcounter{bean}

\newcommand{\qqed}{\hfill\Box}

\title{Pullbacks of Sphere Fibrations over Connected Sums}

\author{Sebastian Chenery}
\address{(Chenery) University of Bristol, School of Mathematics, Fry Building, Woodland Road, Bristol, BS8 1UG}
\email{seb.chenery@bristol.ac.uk}
\author{Stephen Theriault}
\address{(Theriault) University of Southampton, Mathematical Sciences, Building 54, Southampton, SO17 1BJ}
\email{s.d.theriault@soton.ac.uk}


\begin{document}

\begin{abstract}
We prove conditions under which the total space of the pullback of a sphere fibration over a connected sum is homotopy equivalent to a connected sum with a gyration. Existing results of this type often depend on geometric methods. We develop new methods based only on homotopy theory, allowing for generalisations from manifolds to Poincar\'{e} Duality complexes and from integral settings to local ones. Several applications are given.   
\end{abstract}

\maketitle


\section{Introduction} 

It is natural to ask which properties of manifolds really require geometry and which depend only on topology or, more specifically, on homotopy theory. Dependence only on homotopy theory has two advantages: one is a generalisation from manifolds to Poincar\'{e} Duality complexes, which are $CW$-complexes whose cohomology has a cap product satisfying Poincar\'{e} Duality, and another is a generalisation to local settings.

In general, let $F\xrightarrow{\alpha} E\xrightarrow{g} M$ be a homotopy fibration in which all spaces have the homotopy type of Poincar\'e Duality complexes. Let \(dim(M)=n\) and \(dim(E)=m\), and taking $N$ to be another Poincar\'e Duality complex of dimension \(n\), form the connected sum $M\#N$. Define the $m$-dimensional complex $E_N$ as the pullback of $g$ and the collapsing map \(p:M\#N\rightarrow M\), giving a homotopy fibration diagram
\begin{equation}\label{dgm:main1}
    \begin{tikzcd}[row sep=1.5em, column sep = 1.5em]
        F \arrow[rr, "\alpha_N"] \arrow[dd, equal] && E_N \arrow[dd] \arrow[rr] && M\#N \arrow[dd, "p"] \\
        \\
        F \arrow[rr, "\alpha"] && E \arrow[rr, "g"] && M.
    \end{tikzcd}
\end{equation}
A natural question arises: to what extent is \(E_N\) a connected sum? 

This has been the subject of much recent study \cites{basu-ghosh, chenfib, duan, huang_theriault_blowups, huangtheriault, galaz-garcia--reiser, js}. Jeffrey and Selick gave a partial answer in the stricter setting of fibre bundles of smooth, oriented manifolds. They constructed a space $X'$ with the property that there is an isomorphism of integral homology groups \(H_k(E_N)\cong H_k(E)\oplus H_k(X')\) for \(k\) in the range $0<k<m$ \cite{js}*{Theorem 3.3}, suggesting that in certain circumstances $E_N\simeq E\conn X$ for some $m$-dimensional manifold $X$ whose puncture has the homotopy type of \(X'\). They showed that there are contexts in which such an $X$ exists and others in which it cannot. When~\(\alpha\) is null homotopic and the rational cohomologies of $E$ and $M$ have more than one generator, the first author showed that there is a rational homotopy equivalence of based loop spaces \(\Omega E_N\simeq\Omega(E\#X)\) for some Poincar\'{e} Duality complex $X$ \cite{chenfib}. 

A particularly important case is when $F$ is a circle and $\alpha$ is null homotopic. Duan~\cite{duan} showed that if $M$ is a simply-connected smooth manifold then $E_{N}$ is diffeomorphic to $E\conn X$, where $X$ is a special type of smooth manifold called a gyration, defined momentarily. Note that $\alpha$ being null homotopic implies that $F$ retracts off $\Omega M$, so if Duan's result is to be generalised to higher dimensional spheres then those spheres must be $H$-spaces. A generalisation was given by Huang and the second author~\cite{huangtheriault} to $F$ being one of $S^{1}$, $S^{3}$ or $S^{7}$. The proofs in both~\cite{duan} and~\cite{huangtheriault} involved explicitly geometric arguments. A purely homotopy theoretical proof was given in~\cite{huang_theriault_stability} for the special case in which \(F\xrightarrow{\alpha} E\xrightarrow{g} M\) is one of the Hopf fibrations \(S^{1}\rightarrow S^{2n+1}\rightarrow\mathbb{C}P^{n}\) or \(S^{3}\rightarrow S^{4n+3}\rightarrow\mathbb{H}P^{n}\), identifying $E_{N}$ up to homotopy equivalence. The purpose of this paper is to develop a more general homotopy theoretic framework that both recovers existing results and extends them to a wide range of new cases.

To describe our results a definition is needed. Given a smooth $n$-dimensional manifold $N$ and an integer $k\geq 2$ the \textit{gyration} of \(N\), written as $\mathcal{G}^{k}_{0}(N)$, is a smooth manifold of dimension $n+k-1$ obtained from a $(k-1,n)$-type surgery on $N\times S^{k-1}$. This surgery may instead be twisted by some homotopy class $\tau\in\pi_{k-1}(\mathrm{SO}(n))$, resulting in the \textit{twisted gyration} $\mathcal{G}^{k}_{\tau}(N)$, which is again a smooth manifold of dimension $(n+k-1)$. Gyrations originally arose in geometric topology in work of Gonz{\'a}lez Acu{\~n}a \cite{gonzalezacuna} and have since been used in many aspects of topology \cites{basu-ghosh,  bosio_meersseman, chenery:fico, ChenTher:gy_stab, duan, huang_theriault_stability, galaz-garcia--reiser, klpt}. 

Instead of the surgery formulation, we follow \cite{ChenTher:gy_stab} by defining gyrations via pushouts. This allows for two  generalisations. First, we are able to replace $N$ by a Poincar\'{e} Duality complex. Second, the twisting $\tau$ results in a self-map of $S^{n-1}\amsrtimes S^{k-1}$, where $A\amsrtimes B=(A\times B)/(\ast\times B)$ is a half-smash product, involving a torsion class \(S^{n+k-2} \rightarrow S^{n-1}\) lying in the stable homotopy group $\pi^{S}_{k-1}$. This is generalised to any self-map \(t_{\delta}\) of $S^{n-1}\amsrtimes S^{2k-1}$ arising from the difference between the identity map and a torsion class~\(\delta\) in $\pi^{S}_{k-1}$. The pushout in all cases produces an $(n+k-1)$-dimensional Poincar\'{e} Duality complex $\mathcal{G}^{k}_{t_{\delta}}(N)$. As our arguments are homotopy theoretic they also allow for statements that hold after localisation. In particular, any odd dimensional sphere is an $H$-space after localising away from $2$, so we may consider local cases where $F\simeq S^{2k-1}$ and $\alpha$ is null homotopic for some $k\geq 1$. 

\subsection*{Assumptions} 
We describe the cases considered and state hypotheses. Suppose there is a homotopy fibration 
\[
    S^{2k-1} \xlongrightarrow{\alpha} E \xlongrightarrow{g} M
\]
in which \(M\) is a \((2k-1)\)-connected $n$-dimensional Poincar\'{e} Duality complex for $k\geq 1$ and $n\geq 4$. Suppose that $\alpha$ is null homotopic (either integrally for $k\in\{1,2,4\}$ or after localising away from~$2$). Generalising the Hopf fibration cases mentioned above, it will be assumed that $H_{2k}(M;\mathbb{Z})\cong\oplus_{i=1}^{r}\mathbb{Z}$ for some $r\geq 1$. Note that Poincar\'{e} Duality then implies that $n\geq 4k$. The null homotopy for $\alpha$ implies that $S^{2k-1}$ corresponds to a generator of $H_{2k}(M;\mathbb{Z})$. If $r=1$ no more is needed, while if $r>1$ the argument depends on the additional assumptions that $n>4k$ and that, rationally, the other generators in $H_{2k}(M;\mathbb{Q})$ have cohomological duals that square to zero. Let $N$ be a simply-connected $n$-dimensional Poincar\'{e} Duality complex and let \(p:M\conn N\rightarrow M\) be the collapse map. Define the space $E_{N}$ and the map $\alpha_{N}$ by the homotopy pullback diagram 
    \begin{equation}\label{introEN} 
        \begin{tikzcd}[row sep=3em, column sep = 3em]
            S^{2k-1} \arrow[d, equal] \arrow[r, "\alpha_N"] & E_N \arrow[r] \arrow[d] & M\#N \arrow[d, "p"] \\
            S^{2k-1} \arrow[r, "\alpha"] & E \arrow[r, "g"] & M. 
        \end{tikzcd}
    \end{equation} 

\newpage

\begin{mainthm}\label{mainthm} 
    Suppose spaces and maps are as in~(\ref{introEN}) and satisfy the Assumptions. Then: 
    \begin{itemize} 
        \item[(i)] there is a homotopy equivalence \(E_{N}\simeq E\conn\mathcal{G}^{2k}_{t_{\delta}}(N)\) for some $\delta\in\pi^{S}_{2k-1}$; 
        \item[(ii)] $\alpha_{N}$ is null homotopic; 
        \item[(iii)] there is a homotopy equivalence \(\Omega(M\conn N)\simeq S^{2k-1}\times\Omega(E\conn\mathcal{G}_{t_{\delta}}^{2k}(N))\). 
    \end{itemize} 
    Further, after localising away from all primes that divide the order of $\pi^{S}_{2k-1}$, there are homotopy equivalences $E_{N}\simeq E\conn\mathcal{G}^{2k}_{0}(N)$ and $\Omega(N\#M)\simeq S^{2k-1}\times\Omega(E\#\mathcal{G}_0^{2k}(N))$.
\end{mainthm} 

In particular, the Main Theorem recovers the case when 
\(F\xrightarrow{\alpha} E\xrightarrow{g} M\) is one of the Hopf fibrations \(S^{1}\rightarrow S^{2n+1}\rightarrow\mathbb{C}P^{n}\) or \(S^{3}\rightarrow S^{4n+3}\rightarrow\mathbb{H}P^{n}\) 
and extends this to any other case when the rank of $H_{2k}(M;\mathbb{Z})$ is one. It recovers Duan's result on circle bundles, up to homotopy equivalence, and extends it from smooth manifolds to Poincar\'{e} Duality complexes, and from circle bundles to $S^{3}$-bundles. Further, after localising away from $2$ so that $S^{2k-1}$ is an $H$-space for $k\not\in\{1,2\}$, it identifies $E_{N}$ in cases that are inaccessible using geometric arguments.

This paper is organised as follows. Section \ref{sec:gy} discusses connected sums and gyrations in terms of pushouts, and reviews some properties.  Section \ref{sec:E_manipulations} defines and establishes properties of several spaces and maps related to the homotopy fibration diagram~(\ref{introEN}) that will be used to gain control over the homotopy type of $E_{N}$. One of these is a space \(\widehat{E}\) that has a splitting property described in Section~\ref{sec:E-hat_htpytype} and another is a map \(\gamma\), whose interaction with the splitting of \(\widehat{E}\) is a large part what makes the overall argument work. Properties of \(\gamma\) are discussed in Section~\ref{sec:gamma via H_2k};  these are delicate and lead to the addition of the extra hypotheses in the $r>1$ case in the above Assumptions. The Main Theorem is proved in Section \ref{sec:results} and applications are provided in Section \ref{sec:applications}. Along the way, it is necessary to develop a version of Mather's Cube Lemma for a juxtaposition of cubes: this is included as Appendix \ref{appendix} and may be of independent interest.

\subsubsection*{Acknowledgement} During preparation of this work, the first author was supported by EPSRC grant EP/W524621/1 and the Heilbronn Institute for Mathematical Research. 

\section{Connected Sums and Gyrations}
\label{sec:gy} 

This section reviews material on connected sums and gyrations that will be used later 
in the paper. 

\subsection{Connected Sums}\label{subsec:connsums}
\hfill
\vspace{5pt}

Let \(M\) be a simply-connected \(n\)-dimensional Poincar\'e Duality complex. Such an $M$ can be given the structure of a $CW$-complex with a single cell in dimension $n$; fix such a $CW$-structure. Let $\overline{M}$ be the $(n-1)$-skeleton  of $M$. Then there is a homotopy cofibration 
\[
    S^{n-1}\xlongrightarrow{f_{M}}\overline{M}\xlongrightarrow{i_{M}} M
\] 
where $f_{M}$ is the attaching map for the $n$-cell of $M$ and $i_{M}$ is the inclusion of 
the $(n-1)$-skeleton.  

If $N$ is another simply-connected $n$-dimensional Poincar\'{e} Duality complex then 
the connected sum $M\conn N$ is formed by deleting the interior of an $n$-disk 
from both $M$ and $N$ and gluing $M\backslash(D^{n})^{\circ}$ and 
$N\backslash(D^{n})^{\circ}$ together along the boundary $S^{n-1}$ of the deleted disks. 
Topologically, this can be reformulated as follows. Observe that 
$M\backslash(D^{n})^{\circ}\simeq\overline{M}$ and under this homotopy 
equivalence the map including the boundary $S^{n-1}$ of the deleted disk into 
$M\backslash(D^{n})^{\circ}$ becomes the attaching map $f_{M}$. The same is true for~$N$. 
The connected sum is then given by the homotopy cofibration 
\begin{equation} 
   \label{connsumcofib} 
   S^{n-1}\xrightarrow{f_{M}\check{+}f_{M}}\overline{M}\vee\overline{N}\rightarrow M\conn N 
\end{equation} 
where $f_{M}\check{+}f_{N}$ is the composite 
\[
    f_{M}\check{+}f_{N}: S^{n-1}\xrightarrow{\sigma} S^{n-1}\vee S^{n-1}\xrightarrow{f_{M}\vee f_{N}} 
    \overline{M}\vee\overline{N}
\] 
with $\sigma$ being the standard suspension comultiplication. This homotopy cofibration for $M\conn N$ can be equivalently described as a homotopy pushout, which we record for future reference.

\begin{lem} 
   \label{lem:connsumpo} 
   Let $M$ and $N$ be simply-connected $n$-dimensional Poincar\'{e} Duality complexes. 
   Then there is a homotopy pushout
   \[
   \begin{tikzcd}[row sep=3em, column sep=3em]
        S^{n-1} \arrow[r, "f_M"] \arrow[d, "f_N"] & \overline{M} \arrow[d] \\
        \overline{N} \arrow[r]  & M\#N
    \end{tikzcd}
    \]
\end{lem}
\vspace{-0.8cm}~$\qqed$\bigskip  

Two useful observations follow from the homotopy cofibration~(\ref{connsumcofib}). 
One is that 
\(\overline{M\#N}\simeq\overline{M}\vee\overline{N}\)
and the other is that collapsing $\overline{N}$ to a point inside $M\conn N$ 
gives a homotopy cofibration 
\[\overline{N}\rightarrow M\#N\rightarrow M.\] 
Also note that if \(M\) and \(N\) have the homotopy type of oriented manifolds then this construction agrees up to homotopy with the usual orientation preserving connected sum.

\subsection{Gyrations}\label{subsec:gyrations}
\hfill
\vspace{5pt}

For \(k\geq2\) take a map \(\tau:S^{k-1}\rightarrow \mathrm{SO}(n)\) and use the standard linear action of \(\mathrm{SO}(n)\) on \(S^{n-1}\) to define the map 
\begin{equation}\label{linearactiont} 
t:S^{n-1}\times S^{k-1}\rightarrow S^{n-1}\times S^{k-1} 
\end{equation} 
by \(t(a, x)=(\tau(x)\cdot a,x)\).  

\begin{defn}\label{def:gy}
    Let \(k\geq2\) be an integer and let \(N\) be an \(n\)-dimensional Poincar\'e Duality complex. Define the \textit{\(k\)-gyration of \(N\) by \(\tau\)} to be the space defined by the (strict) pushout
        \begin{equation}\label{gyrationpo}
            \begin{tikzcd}[row sep=3em, column sep=3em]
                S^{n-1}\times S^{k-1} \arrow[r, "1\times \iota"] \arrow[d, "(f_N\times 1)\circ t"] & S^{n-1}\times D^k \arrow[d] \\
                \overline{N}\times S^{k-1} \arrow[r] & \mathcal{G}^k_\tau(N) 
            \end{tikzcd}
        \end{equation} 
    where $\iota$ is the inclusion of the boundary of the disc. When the context is clear, we will usually just write \textit{gyration} for \(\mathcal{G}^k_\tau(N)\). 
\end{defn} 

An important special case is when $\tau$ is the trivial map, in which case $t$ is the identity map and the gyration is written as $\mathcal{G}_{0}^{k}(N)$. We call this the trivial \(k\)-gyration. 

\begin{rem}
    The gyration \(\mathcal{G}^{k}_{\tau}(N)\) is a Poincar\'{e} Duality complex of dimension \(n+k-1\). Furthermore, via the alternative surgery definition (see for example \cite{huang_inertness24}*{Section 12}) it follows that if \(N\) is an oriented manifold then \(\mathcal{G}^{k}_{\tau}(N)\) is an $(n+k-1)$-manifold with an orientation inherited from that of \(N\). 
\end{rem} 

The definition of a gyration may be refined, from a homotopy theoretic perspective. To set up, 
for two path-connected, based spaces $X$ and $Y$, the \textit{(right) half-smash} of $X$ and $Y$ is the quotient space \(X\amsrtimes Y=X\times Y/\{x_0\}\times Y\), where $x_0$ is the basepoint of $X$. There is a homotopy cofibration 
\[
    Y\xrightarrow{i_2}X\times Y\rightarrow X\amsrtimes Y
\] 
where \(i_2\) is the inclusion of the second factor. It is well known that if $X$ is a co-$H$-space, then there is a homotopy equivalence $X\amsrtimes Y\simeq X\vee(X\wedge Y)$.
In our case, consider the homotopy cofibration 
\[
    S^{k-1}\xrightarrow{i_{2}} S^{n-1}\times S^{k-1}\xrightarrow{p} S^{n-1}\amsrtimes S^{k-1}
\] 
where $i_{2}$ is the inclusion of the second factor and $p$ is the quotient map. Observe that the definition of~$t$ in~(\ref{linearactiont}) implies that $t\circ i_{2}=i_{2}$. Therefore there is a homotopy cofibration diagram 
\begin{equation} 
  \label{t'def} 
    \begin{tikzcd}[row sep=3em, column sep = 3em]
        S^{k-1} \arrow[d, equal] \arrow[r, "i_{2}"] & S^{n-1}\times S^{k-1} \arrow[d, "t"] \arrow[r, "p"] &  S^{n-1}\amsrtimes S^{k-1} \arrow[d, "t'"] \\ 
        S^{k-1} \arrow[r, "i_{2}"] & S^{n-1}\times S^{k-1} \arrow[r, "p"] & S^{n-1}\amsrtimes S^{k-1}
    \end{tikzcd}
\end{equation} 
that defines the map $t'$ as an induced map of homotopy cofibres. Since $t$ is a homotopy 
equivalence, it induces an isomorphism in homology, and therefore the five-lemma applied 
to the long exact sequences in homology induced by the homotopy cofibration diagram~(\ref{t'def}) 
imply that $t'$ also induces an isomorphism in homology. Since spaces are simply-connected, 
$t'$ is therefore a homotopy equivalence by Whitehead's Theorem. 

Returning to the definition of a gyration in~(\ref{def:gy}), as $D^{k}$ is contractible the restrictions of both horizontal arrows to $S^{k-1}$ are null homotopic, so this space may be collapsed out to give a homotopy pushout 
\begin{equation}\label{gyrationpo2temp}
    \begin{tikzcd}[row sep=3em, column sep=3em]
        S^{n-1}\amsrtimes S^{k-1} \arrow[r, "\pi"] \arrow[d, "(f_N\amsrtimes 1)\circ t'"] & S^{n-1} \arrow[d] \\
        \overline{N}\amsrtimes S^{k-1} \arrow[r] & \mathcal{G}^{k}_\tau(N)  
    \end{tikzcd}
\end{equation} 
where $t$ has been replaced by the map $t'$ in~(\ref{t'def}). 
This homotopy pushout could equally be regarded as defining $\mathcal{G}^{k}_{\tau}(N)$ up to homotopy. From that point of view, the definition of a gyration could be extended, replacing the self-equivalence $t'$ of $S^{n-1}\amsrtimes S^{k-1}$ that arises from $\tau$ with any self-equivalence that results in the homotopy pushout satisfying Poincar\'{e} Duality.  

To that end, observe that there is a homotopy cofibration 
\[
    S^{n-1}\xrightarrow{i} S^{n-1}\amsrtimes S^{k-1}\xrightarrow{q} S^{n+k-2}
\] 
where $i$ is the inclusion of the first wedge summand and $q$ is the quotient map. Assuming $n>k$, an element in $\pi^{S}_{k-1}$ can be represented by a map \(\delta: S^{n+k-2}\rightarrow S^{n-1}\). Let $\epsilon_{\delta}$ be the composite 
\begin{equation}\label{def:epsilondelta} 
   \epsilon_{\delta}: S^{n-1}\amsrtimes S^{k-1}\xrightarrow{q} S^{n+k-2}\xrightarrow{\delta} S^{n-1}\xrightarrow{i} S^{n-1}\amsrtimes S^{k-1}. 
 \end{equation}  
Since $S^{n-1}\amsrtimes S^{k-1}\simeq S^{n-1}\vee S^{n+k-2}$, it is a co-$H$-space. Using the comultiplication, define the map 
\begin{equation}\label{def:tdelta} 
   t_{\delta}: S^{n-1}\amsrtimes S^{k-1}\rightarrow S^{n-1}\amsrtimes S^{k-1} 
\end{equation}  
by $t_{\delta}=1-\epsilon_{\delta}$. 

\begin{lem}\label{tdeltaequiv} 
   For any $\delta\in\pi^{S}_{k-1}$ and $n>k\geq 2$, the map $t_{\delta}$ is a homotopy equivalence with the 
   property that it induces the identity map in homology and cohomology.  
\end{lem} 

\begin{proof} 
Observe that as $k\geq 2$, we have 
$n+k-2>n-1$, implying that $\delta$ induces the zero map in homology and cohomology. 
Thus $(\epsilon_{\delta})_{\ast}=0$ and $(\epsilon_{\delta})^{\ast}=0$, implying that  
$(t_{\delta})_{\ast}=1_{\ast}$ and $(t_{\delta})^{\ast}=1^{\ast}$. In particular, as $t_{\delta}$ induces an isomorphism in homology and \(S^{n-1}\amsrtimes S^{k-1}\) is simply-connected for \(n>k\geq2\), it is a homotopy equivalence by Whitehead's Theorem. 
\end{proof}

Define the gyration $\mathcal{G}^{k}_{t_{\delta}}(N)$ by the homotopy pushout 
\begin{equation}\label{gyrationpotdelta}
    \begin{tikzcd}[row sep=3em, column sep=3em]
        S^{n-1}\amsrtimes S^{k-1} \arrow[r, "\pi"] \arrow[d, "(f_N\amsrtimes 1)\circ t_{\delta}"] & S^{n-1} \arrow[d] \\
        \overline{N}\amsrtimes S^{k-1} \arrow[r] & \mathcal{G}^{k}_{t_{\delta}}(N).  
    \end{tikzcd}
\end{equation}
Note that as $t_{\delta}$ induces the identity map in cohomology, there is a isomorphism 
$H^{\ast}(\mathcal{G}^{k}_{t_{\delta}}(N))\cong H^{\ast}(\mathcal{G}^{k}_{0}(N))$ (as rings), implying that $\mathcal{G}^{k}_{t_{\delta}}(N)$ is a Poincar\'{e} Duality complex. 

Basu-Ghosh~\cite{basu-ghosh}*{Proposition 6.9} identified the homotopy type of of the \((n+k-2)\)-skeleton of \(\mathcal{G}^k_\tau(N)\) using a geometric approach. A different formulation was given in~\cite{ChenTher:gy_stab}*{Lemma 3.2} that also identified the homotopy class of the attaching map for the top cell of $\mathcal{G}^{k}_{\tau}(N)$. The latter argument also applies to $\mathcal{G}^{k}_{t_{\delta}}(N)$, as will now be discussed. Let \(\pi:S^{n-1}\amsrtimes S^{k-1}\rightarrow S^{n-1}\) be the natural projection map. The homotopy fibre of $\pi$ is homotopy equivalent to \mbox{$\Sigma\Omega S^{n-1}\wedge S^{k-1}$} (see~\cite{selick}*{Theorem 7.7.3}), implying that its $(n+k-2)$-skeleton is homotopy equivalent to $S^{n+k-2}$. Let $j$ be the composite 
\[
    j: S^{n+k-2}\xhookrightarrow{} \Sigma\Omega S^{n-1}\wedge S^{k-1}\rightarrow S^{n-1}\amsrtimes S^{k-1}.
\] 
In general, given maps \(f: A\rightarrow C\) and \(g: B\rightarrow C\), let \(f\perp g: A\vee B\rightarrow C\) be their \textit{wedge sum}. The following lemma was proved in~\cite{ChenTher:gy_stab}*{Lemma 2.2}. 

\begin{lem}\label{lem:spherehalfsmash}
    Let \(q,i,j\) and \(\pi\) be defined as above. If $k\leq n-1$ then
    \begin{enumerate}
        \item[(i)] there is a homotopy cofibration \(S^{n+k-2}\xrightarrow{j} S^{n-1}\amsrtimes S^{k-1}\xrightarrow{\pi} S^{n-1}\);
        \item[(ii)] the map \(q\) is a left homotopy inverse for \(j\); 
        \item[(iii)] the map \(j\) is a co-$H$-map;
        \item[(iv)] there is a homotopy equivalence \({i\perp j}: S^{n-1}\vee S^{n+k-2}\rightarrow S^{n-1}\amsrtimes S^{k-1}\).
    \end{enumerate}  
\end{lem} 
\vspace{-0.8cm}~$\qqed$\bigskip 

Now consider the diagram 
\[\begin{tikzcd}[row sep=3em, column sep=3em] 
        S^{n+k-2}\arrow[r]\arrow[d, "j"] & \ast\arrow[d] \\ 
        S^{n-1}\amsrtimes S^{k-1} \arrow[r, "\pi"] \arrow[d, "(f_N\amsrtimes 1)\circ t_{\delta}"] & S^{n-1} \arrow[d] \\
        \overline{N}\amsrtimes S^{k-1} \arrow[r] & \mathcal{G}^{k}_{t_{\delta}}(N).  
    \end{tikzcd} 
\] 
The upper square is a homotopy pushout by Lemma~\ref{lem:spherehalfsmash}~(i) and the lower square is a homotopy pushout by~(\ref{gyrationpotdelta}). Thus the outer rectangle is also a homotopy pushout. This implies that there is a homotopy cofibration 
\[ 
   S^{n+k-2} \xlongrightarrow{\varphi_{t_{\delta}}} \overline{N}\amsrtimes S^{k-1} \longrightarrow \mathcal{G}^{k}_{t_{\delta}}(N)
\] 
where
\[
    \varphi_{t_{\delta}}: S^{n+k-2}\xrightarrow{j} 
    S^{n-1}\amsrtimes S^{k-1}\xrightarrow{t_{\delta}} S^{n-1}\amsrtimes S^{k-1}\xrightarrow{f_{N}\amsrtimes 1} \overline{N}\amsrtimes S^{k-1} 
\] 
attaches the top cell to $\mathcal{G}^{k}_{t_{\delta}}(N)$. This proves the following.  
   
\begin{lem} 
   \label{lem:gyrationskel} 
   The \((n+k-2)\)-skeleton of \(\mathcal{G}^k_{t_{\delta}}(N)\) is homotopy equivalent to the half-smash \(\overline{N}\amsrtimes S^{k-1}\) and the map $\varphi_{t_{\delta}}:S^{n+k-2}\rightarrow\overline{N}\amsrtimes S^{k-1}$ is the attaching map for the top cell.~$\qqed$ 
\end{lem} 

In particular, if $t$ is the identity map, implying that $t_{\delta}$ is the identity map, then the attaching map for the top cell of $\mathcal{G}_{0}^{k}(N)$ is the composite 
\begin{equation}\label{eqn:G0attach} 
\varphi_{0}: S^{n+k-2}\xrightarrow{j} 
    S^{n-1}\amsrtimes S^{k-1}\xrightarrow{f_{N}\amsrtimes 1} \overline{N}\amsrtimes S^{k-1}. 
\end{equation}

\begin{rem}
Note that the homotopy equivalence for the $(n+k-2)$-skeleton of $\mathcal{G}^{k}_{t_{\delta}}(N)$ in Lemma~\ref{lem:gyrationskel} does not depend on the choice of \(t_{\delta}\), but that the homotopy class of the attaching map $\varphi_{t_{\delta}}$ does.  
\end{rem}

\section{Properties of a Family of Homotopy Pullbacks}
\label{sec:E_manipulations} 

In this section, we establish a framework of spaces and maps designed to gain control over the primary space of interest, the pullback $E_N$ discussed in the Introduction. Let $k\geq 1$ be an integer. Suppose that there is a homotopy fibration 
\begin{equation}\label{eq:htpyfib}\tag{\(\mathfrak{F}\)}
    S^{2k-1} \xlongrightarrow{\alpha} E \xlongrightarrow{g} M
\end{equation}
in which \(M\) is a \((2k-1)\)-connected $n$-dimensional Poincar\'{e} Duality complex and $\alpha$ is null homotopic. Note that $\alpha$ being null homotopic implies that the homotopy fibration connecting map \(\Omega M\rightarrow S^{2k-1}\) has a right homotopy inverse. Consequently, $S^{2k-1}$ is an $H$-space. Integrally this holds only if $k\in\{1,2,4\}$ and locally, after inverting 2, this holds for any positive $k$. Note also this implies that the homotopy fibre in (\ref{eq:htpyfib}) cannot be a sphere of even dimension; that is, the odd dimension is forced. 
 
Since $E$ is the total space of a homotopy fibration where both the fibre and the base are Poincar\'{e} Duality complexes of dimensions \(2k-1\) and \(n\) respectively, a theorem of Quinn~\cite{quinn} implies that $E$ is a Poincar\'{e} Duality complex of dimension $n+2k-1$. 

Let $N$ be another simply-connected Poincar\'{e} Duality complex of dimension $n$ and form the connected sum $M\conn N$. Recall from Section \ref{sec:gy} that the $(n-1)$-skeleton of $M\conn N$ is $\overline{M}\vee\overline{N}$ and that by collapsing out~$\overline{N}$ we obtain a canonical map \(M\# N \rightarrow M\). Moreover, observe that the composite \(\overline{M}\rightarrow M\# N\rightarrow M\) is identical to the inclusion of the $(n-1)$-skeleton of \(M\). Pulling back (\ref{eq:htpyfib}) with respect to this composite, define spaces $\widehat{E}$ and $E_{N}$ and maps $\widehat{\alpha}$ and $\alpha_{N}$ by the iterated homotopy pullback diagram
\begin{equation}\label{dgm:null}
    \begin{tikzcd}[row sep=3em, column sep = 3em]
        S^{2k-1} \arrow[r, "\widehat{\alpha}"] \arrow[d, equal] & \widehat{E} \arrow[r, "\widehat{g}"] \arrow[d] & \overline{M} \arrow[d]  \\
        S^{2k-1} \arrow[d, equal] \arrow[r, "\alpha_N"] & E_N \arrow[r, "g_N"] \arrow[d] & M\#N \arrow[d] \\
        S^{2k-1} \arrow[r, "\alpha"] & E \arrow[r, "g"] & M.
    \end{tikzcd}
\end{equation} 

The goal, achieved in Theorem~\ref{thm:main}, is to identify the homotopy type of $E_{N}$. We begin by establishing some initial properties of the spaces and maps in (\ref{dgm:null}). Define $\overline{E}$ as the $(n+2k-2)$-skeleton of $E$.

\begin{prop} 
   \label{prop:skeltoskel} 
    There is a homotopy commutative diagram 
    \begin{equation*}
        \begin{tikzcd}[row sep=3em, column sep=3em]
            \overline{E} \arrow[r, "\overline{g}"] \arrow[d] & \overline{M} \arrow[d] \\
            E \arrow[r, "g"] & M
        \end{tikzcd}
    \end{equation*}
    for some map $\overline{g}$, where the vertical maps are the skeletal inclusions. 
\end{prop} 

\begin{proof} 
Consider the integral cohomology Serre spectral sequence for the homotopy fibration 
\[
    S^{2k-1} \xlongrightarrow{\alpha} E \xlongrightarrow{g} M
\]
that converges to $H^{\ast}(E)$. The $E_{2}$-page is $H^{\ast}(S^{2k-1})\otimes H^{\ast}(M)$. Let $x\in H^{2k-1}(S^{2k-1})$ be a generator. Then the $E_{2}$-page is concentrated in two rows and has elements of the form $1\otimes a$ and $x\otimes a$ for $a\in H^{\ast}(M)$. Since $M$ is $(2k-1)$-connected, $H^{m}(M)\cong 0$ for $1\leq m<2k$. By Poincar\'{e} Duality, $H^{m}(M)\cong 0$ for $n-2k<m<n$ and $H^{n}(M)\cong\varmathbb{Z}$. Let $w\in H^{n}(M)$ represent the generator.
\medskip 
 
\textit{Step 1:} We show that $d^{2k}(x)=y$ for some $y\in H^{2k}(M)$. Since the $E_{2}$-page 
is concentrated in two rows separated by $2k-1$ rows of zeroes, for degree reasons the only possible non-trivial differential is~$d^{2k}$. Since the map \(S^{2k-1}\rightarrow E\) is null homotopic, it induces the zero map in cohomology. Therefore the element $x\otimes 1$ in the $E_{2}$-page cannot survive the spectral sequence, implying that it is in the domain of a differential. As the only possible differential is $d^{2k}$ we obtain $d^{2k}(x)=y$ for some $y\in H^{2k}(M)$. Note that as~$M$ is $(2k-1)$-connected, the Universal Coefficient Theorem implies that $H^{2k}(M)$ is torsion-free, so $y$ generates a $\mathbb{Z}$-summand. 
\medskip  

\textit{Step 2:} We show that $w$ is in the image of $d^{2k}$. As $d^{2k}$ is a derivation and $d^{2k}(a)=0$ for all $a\in H^{2k}(M)$, the fact that $d^{2k}(x)=y$ implies that $d^{2k}(x\otimes a)=y\otimes a$ for any $a\in H^{\ast}(M)$. By Poincar\'{e} Duality, as~$y$ generates a $\mathbb{Z}$-summand in $H^{2k}(M)$, there is a class $z\in H^{n-2k}(M)$ such that  $y\cup z=w$. Then $d^{2k}(x\otimes z)=y\otimes z=y\cup z=w$. 
\medskip 

\textit{Step 3:} Consider the elements of highest degree in $H^{\ast}(E)$. Observe that the elements of highest degree in $H^{\ast}(M)$ occur in degrees $n-2k$ and $n$. Any 
element $v\in H^{n-2k}(M)$ results in an element $x\otimes v$ in degree $n-1$ on the 
$E_{2}$-page. Thus the only elements of degree~$\geq n$ on the $E^{2}$ page are  
$1\otimes w$ and $x\otimes w$ in degrees $n$ and $n+2k-1$ respectively. By 
Step 2, $1\otimes w=w$ does not survive the spectral sequence while $x\otimes w$ does survive the spectral sequence for degree reasons. Thus $H^{\ast}(E)$ has only one element of degree~$\geq n$, which occurs in degree $n+2k-1$. 
\medskip 

\textit{Step 4:} Bringing the above arguments together, by Step 3 the $(n+2k-2)$-skeleton $\overline{E}$ of $E$ has dimension~$<n$. Therefore the composite \(\overline{E} \rightarrow E \xrightarrow{g} M\) factors through the $(n-1)$-skeleton $\overline{M}$ of $M$. This gives the homotopy commutative diagram asserted by the proposition. 
\end{proof} 

\begin{lem} \label{lem:Ehatproperty2} 
    If $\overline{M}$ is not contractible then the maps \(\widehat{\alpha}:S^{2k-1} \rightarrow \widehat{E}\) 
    and \(\alpha_{N}:S^{2k-1}\rightarrow{E_N}\) in~(\ref{dgm:null}) are null homotopic. 
\end{lem} 

\begin{proof} 
First consider the skeletal inclusion \(\overline{M} \rightarrow M\). The homotopy cofibration \(S^{n-1}\rightarrow \overline{M} \rightarrow M\)  implies that there is a homotopy fibration \(F \rightarrow \overline{M} \rightarrow M\) where $F$ is $(n-2)$-connected. By hypothesis, $M$ is $(2k-1)$-connected and $n$-dimensional. Since $\overline{M}$ is not contractible, it too is $(2k-1)$-connected and has a nonzero cohomology class in degree~$\geq 2k$. By Poincar\'{e} Duality, this implies that $M$ is at least $4k$-dimensional. Thus $n\geq 4k$. Therefore the skeletal inclusion \(\overline{M} \rightarrow M\) is a homotopy equivalence in dimensions $\leq 4k-2$. 

Next, observe that the outer rectangle in~(\ref{dgm:null}) gives a diagram of homotopy fibration sequences  
\begin{equation} 
  \label{2partials} 
    \begin{tikzcd}[row sep=3em, column sep=3em]
        \Omega\overline{M} \arrow[r, "\overline{\partial}"] \arrow[d] & S^{2k-1} \arrow[r, "\widehat{\alpha}"] \arrow[d, equal] 
        & \widehat{E} \arrow[r] \arrow[d] 
        & \overline{M} \arrow[d] \\
        \Omega M \arrow[r, "\partial"] 
        & S^{2k-1} \arrow[r, "\alpha"] & E \arrow[r] & M\conn N
    \end{tikzcd}
\end{equation} 
where $\partial$ and $\overline{\partial}$ are the homotopy fibration connecting maps. 
Since $\alpha$ is null homotopic, $\partial$ has a right homotopy inverse \(s: S^{2k-1}\rightarrow \Omega M\). 
Since the skeletal inclusion \(\overline{M} \rightarrow M\) is a homotopy equivalence in dimensions $\leq 4k-2$, the map \(\Omega\overline{M} \rightarrow \Omega M\) is a homotopy equivalence in dimensions $\leq 4k-3$. 
Thus~$s$ lifts to a map \(\overline{s}:S^{2k-1} \rightarrow \Omega\overline{M}\). The homotopy commutativity of the left square in~(\ref{2partials}) then implies that $\overline{\partial}\circ\overline{s}\simeq\partial\circ s$, implying in turn that $\overline{s}$ is a right homotopy inverse for $\overline{\partial}$. Therefore in the homotopy fibration
\[
    \Omega\overline{M} \xlongrightarrow{\overline{\partial}} S^{2k-1} \xlongrightarrow{\widehat{\alpha}} \widehat{E}
\]
the map $\widehat{\alpha}$ is null homotopic. Finally, as $\widehat{\alpha}$ is null homotopic, the 
homotopy commutativity of the upper left square in~(\ref{dgm:null}) implies that $\alpha_N$ is also null homotopic.
\end{proof} 

What follows is the main result of this section, which identifies $E_{N}$ as a homotopy pushout. For maps \(a:\Sigma X \rightarrow Z\) and \(b:\Sigma Y  \rightarrow Z\), let \([a,b]:\Sigma X\wedge Y \rightarrow Z\) be their Whitehead product. 

\begin{prop}\label{prop:betagamma}
    There is a diagram of iterated homotopy pushouts
    \begin{equation*} 
        \begin{tikzcd}[row sep=3em, column sep=3em]
            S^{n-1}\amsrtimes S^{2k-1} \arrow[r, "f_N\amsrtimes 1"] \arrow[d, "\gamma"] & \overline{N}\amsrtimes S^{2k-1} \arrow[r] \arrow[d, "\beta"] & \ast \arrow[d]  \\
            \widehat{E} \arrow[r] & E_N \arrow[r] & E
        \end{tikzcd}
    \end{equation*}
    for some maps \(\beta\) and \(\gamma\), where: 
    \begin{itemize} 
    \item[(i)] $\gamma$ satisfies a homotopy commutative diagram
    \begin{equation*} 
        \begin{tikzcd}[row sep=3em, column sep=3em]
           S^{n-1} \arrow[r, "i"] \arrow[dr, equal] & S^{n-1}\amsrtimes S^{2k-1} \arrow[r, "\gamma"] \arrow[d, "\pi"] & \widehat{E} \arrow[d, "\widehat{g}"]  \\
           & S^{n-1} \arrow[r, "f_M"] & \overline{M};
        \end{tikzcd}
    \end{equation*} 
    \item[(ii)] if $S^{2k-1}$ has a classifying space 
    $BS^{2k-1}$ (integrally if $k\in\{1,2\}$ or rationally for any $k\geq 1$) and \(g:E \rightarrow M\) is induced by a map \(\zeta: M \rightarrow BS^{2k-1}\), then $\gamma$ can be chosen to satisfy part~(i) as well as a homotopy commutative diagram (integrally if $k\in\{1,2\}$ or rationally for any $k\geq 1$) 
    \begin{equation*}
        \begin{tikzcd}[row sep=3em, column sep = 6em]
                S^{n-1}\amsrtimes S^{2k-1} \arrow[d, "\simeq"] \arrow[r, "\gamma"] & \widehat{E} \arrow[d, "\widehat{g}"] \\
                S^{n-1}\vee S^{n+2k-2} \arrow[r, "f_M\perp {[f_M,s]}"] & \overline{M}
        \end{tikzcd}
    \end{equation*} 
    where $s$ is the composite \(S^{2k-1} \xrightarrow{\widehat{s}} \Omega\overline{M} \xrightarrow{ev} \overline{M}\) of the evaluation map $ev$ and a right homotopy inverse $\widehat{s}$ for 
    \[
        \Omega\overline{M}\rightarrow\Omega M\xrightarrow{\Omega\zeta} S^{2k-1}.
    \] 
    \end{itemize} 
\end{prop}

\begin{proof}
Consider the diagram 
\begin{equation*}
\begin{tikzcd}[row sep=3em, column sep=3em]
     S^{n-1} \arrow[r, "f_{N}"] \arrow[d, "f_{M}"] & \overline{N} \arrow[r] \arrow[d] & \ast\arrow[d]  \\
     \overline{M} \arrow[r] & M\conn N \arrow[r] & M.  
\end{tikzcd} 
\end{equation*} 
The left square is the homotopy pushout defining $M\conn N$.  The right square is a homotopy pushout  since~$M$ is the homotopy cofibre of the map \(\overline{N} \rightarrow M\conn N\). Thus this is a diagram of iterated homotopy pushouts. Using the map \(g:E \rightarrow M\), apply Theorem~\ref{homotopycube} to obtain a homotopy commutative cube
\begin{equation}\label{FGHcube}
    \begin{tikzcd}[row sep=1em, column sep=1em]
        F \arrow[rr] \arrow[dd] \arrow[dr] && G \arrow[rr] \arrow[dr] \arrow[dd, near start, swap] &&
        H \arrow[dd] \arrow[dr] \\
        & \widehat{E} \arrow[rr, crossing over] && E_{N} \arrow[rr, crossing over] && E \arrow[dd] \\
        S^{n-1} \arrow[rr, near start, "f_N"] \arrow[dr, swap, "f_M"] && \overline{N} \arrow[rr] \arrow[dr] && \ast\arrow[dr]\\
        & \overline{M} \arrow[rr] \arrow[from=uu, crossing over] && M\#N \arrow[rr] \arrow[from=uu, crossing over] && M \\
    \end{tikzcd}
\end{equation}
where the sides of each cube are homotopy pullbacks, these pullbacks define $F$, $G$ and $H$, and the top face of each cube is a homotopy pushout.  

We wish to identify $F$, $G$ and $H$ and the maps between them, up to homotopy equivalence. To do this, label the maps \(H\rightarrow \ast\), \(G\rightarrow \overline{N}\) and \(F\rightarrow S^{n-1}\) in~(\ref{FGHcube}) as $a$, $b$ and $c$ respectively. We will construct a homotopy commutative diagram  
\begin{equation}\label{dgm:rectangle}   
    \begin{tikzcd}[row sep=3em, column sep=3em]
       S^{n-1}\times S^{2k-1} \arrow[r, "f_N\times 1"] \arrow[d, "c'"] & \overline{N}\times S^{2k-1} \arrow[r, "\pi_{2}"] \arrow[d, "b'"] & S^{2k-1} \arrow[d, "a'"]  \\
       F \arrow[r] \arrow[d, "c"] & G \arrow[r] \arrow[d, "b"] & H \arrow[d,  "a"] \\ 
       S^{n-1} \arrow[r, "f_{N}"] & \overline{N} \arrow[r] & \ast  
    \end{tikzcd}
\end{equation} 
where the two upper squares are homotopy pullbacks and the maps $a'$, $b'$ and $c'$ are homotopy equivalences. 
First consider the face $H$-$E$-$\ast$-$M$ in~(\ref{FGHcube}). As this face is a homotopy pullback and there is a homotopy fibration  \(S^{2k-1}\xrightarrow{\alpha} E\rightarrow M\),  there is a homotopy equivalence \(a': S^{2k-1}\rightarrow H\) such that the composite \(S^{2k-1}\xrightarrow{a'} H\rightarrow E\) is $\alpha$. Second, consider the face $G$-$H$-$\overline{N}$-$\ast$ in~(\ref{FGHcube}). As this is a homotopy pullback over a point, there is a homotopy equivalence \(b':\overline{N}\times S^{2k-1}\rightarrow G\) 
that makes the upper right square in~(\ref{dgm:rectangle}) homotopy commute, and for which $b\circ b'\simeq\pi_{1}$. Note that as $a'$ and $b'$ are homotopy equivalences, the upper right square in~(\ref{dgm:rectangle}) is a homotopy pullback. Third, consider the face $F$-$G$-$S^{n-1}$-$\overline{N}$ in~(\ref{FGHcube}).  Since $b\circ b'\simeq\pi_{1}$, there is a homotopy pullback 
\begin{equation*} 
    \begin{tikzcd}[row sep=3em, column sep=3em]
        S^{n-1}\times S^{2k-1} \arrow[r, "f_{N}\times 1"] \arrow[d, "\pi_{1}"] & \overline{N}\times S^{2k-1} \arrow[d,"b\circ b'"] \\
        S^{n-1} \arrow[r, "f_{N}"] & \overline{N}.  
    \end{tikzcd}
\end{equation*} 
Since $F$-$G$-$S^{n-1}$-$\overline{N}$ is a homotopy pullback, the maps $b'\circ(f_{N}\times 1)$ and $\pi_{1}$ induce a homotopy pullback map 
\(c': S^{n-1}\times S^{2k-1}\rightarrow F\) 
that makes the upper left square in~(\ref{dgm:rectangle}) homotopy commute, and for which $c\circ c'\simeq\pi_{1}$. Since the lower left square and left rectangle in~(\ref{dgm:rectangle}) are homotopy pullbacks, so is the upper left square. The fact that $b'$ is a homotopy equivalence then implies that $c'$ is a homotopy equivalence. Hence we have constructed the two upper squares in~(\ref{dgm:rectangle}) as asserted. 

Juxtapose the upper squares in~(\ref{dgm:rectangle}) with the two top faces in~(\ref{FGHcube}) to obtain a homotopy commutative diagram 
\begin{equation}\label{dgm:rectanglepo}   
    \begin{tikzcd}[row sep=3em, column sep=3em]
       S^{n-1}\times S^{2k-1} \arrow[r, "f_N\times 1"] \arrow[d, "c'"] & \overline{N}\times S^{2k-1} \arrow[r, "\pi_{2}"] \arrow[d, "b'"] & S^{2k-1} \arrow[d, "a'"]  \\
       F \arrow[r] \arrow[d] & G \arrow[r] \arrow[d] & H \arrow[d] \\ 
       \widehat{E} \arrow[r] & E_{N} \arrow[r] & E.  
    \end{tikzcd}
\end{equation} 
Since the two lower squares are homotopy pushouts and the two upper squares homotopy commute with $c'$, $b'$ and~$a'$ being homotopy equivalences, the left and right rectangles are homotopy pushouts. By definition of $a'$, the composite 
\(S^{2k-1}\xrightarrow{a'} H\rightarrow E\) 
is $\alpha$. The pullbacks defining $\widehat{E}$ and $E_{N}$ in~(\ref{dgm:null}) then imply that the restriction of the left and middle columnns in~(\ref{dgm:rectanglepo}) to $S^{2k-1}$ are $\widehat{\alpha}$ and $\alpha_{N}$ respectively. The null homotopy for $\widehat{\alpha}$ in Lemma~\ref{lem:Ehatproperty2} then implies there are compatible null homotopies for $\alpha_{N}$ and $\alpha'$, and therefore the copy of $S^{2k-1}$ in the homotopy pushouts given by the left and right rectangles in~(\ref{dgm:rectanglepo}) 
can be collapsed out, giving a diagram of homotopy pushouts  
\begin{equation} \label{dgm:betagamma}
    \begin{tikzcd}[row sep=3em, column sep=3em]
       S^{n-1}\amsrtimes S^{2k-1} \arrow[r, "f_N\amsrtimes 1"] \arrow[d, "\gamma"] & \overline{N}\amsrtimes S^{2k-1} \arrow[r] \arrow[d, "\beta"] & \ast \arrow[d]  \\
       \widehat{E} \arrow[r] & E_N \arrow[r] & E
    \end{tikzcd}
\end{equation}
where $\gamma$ and $\beta$ are quotient maps. This is the iterated homotopy pushout asserted to exist by the proposition.  
\medskip 

\textit{Property (i):} As the homotopy equivalence $c'$ identified the map \(F \rightarrow S^{n-1}\) in~(\ref{FGHcube}) as the projection map \(\pi_1:S^{n-1}\times S^{2k-1} \rightarrow S^{n-1}\),
the left-most face of the commutative cube (\ref{FGHcube}) has been identified as a homotopy pullback 
\begin{equation*} 
    \begin{tikzcd}[row sep=3em, column sep=3em]
        S^{n-1}\times S^{2k-1} \arrow[r] \arrow[d, "\pi_1"] & \widehat{E} \arrow[d, "\widehat{g}"]  \\
        S^{n-1} \arrow[r, "f_M"] & \overline{M}. 
    \end{tikzcd}
\end{equation*} 
Here, the upper horizontal map is the composite \(c'':S^{n-1}\times S^{2k-1} \xrightarrow {c'} F \rightarrow \widehat{E}\),  so as the restriction of this composite to $S^{2k-1}$ is null homotopic and collapsing out \(S^{2k-1}\) defined \(\gamma\), we obtain a homotopy commutative square
\begin{equation*} 
    \begin{tikzcd}[row sep=3em, column sep=3em]
           S^{n-1}\amsrtimes S^{2k-1} \arrow[r, "\gamma"] \arrow[d, "\pi"] & \widehat{E} \arrow[d, "\widehat{g}"]  \\
           S^{n-1} \arrow[r, "f_M"] & \overline{M}.
    \end{tikzcd}
\end{equation*} 
Since $\pi$ is a left inverse for \(i:S^{n-1} \rightarrow S^{n-1}\amsrtimes S^{2k-1}\), we obtain $f_{M}\simeq f_{M}\circ\pi\circ i\simeq\widehat{g}\circ\gamma\circ i$, as asserted. 
\medskip 

\textit{Property (ii):} Let $\widehat{\zeta}$ be the composite
\[
    \widehat{\zeta}:\overline{M}\longrightarrow M\xlongrightarrow{\zeta} BS^{2k-1}.
\]
The definition of $\widehat{E}$ has a homotopy pullback in~(\ref{dgm:null}) implies that the homotopy fibration 
\[
    S^{2k-1}\xlongrightarrow{\widehat{\alpha}}\widehat{E}\xlongrightarrow{\widehat{g}}\overline{M}
\] 
is induced by $\widehat{\zeta}$. By Lemma~\ref{lem:Ehatproperty2}, $\widehat{\alpha}$ is null homotopic, so $\Omega\widehat{\zeta}$ has a right homotopy inverse, which we shall write as \(\widehat{s}:S^{2k-1} \rightarrow \Omega\overline{M}\).

The map \(\gamma:S^{n-1}\amsrtimes S^{2k-1} \rightarrow \widehat{E}\) was defined as an extension that existed because the restriction of \(c''\) to $S^{2k-1}$ is null homotopic, but note that there could be different choices of such an extension. In~\cite{bt2}*{Theorem 3.4} it was shown that, because \(\widehat{E} \rightarrow \overline{M}\) is induced by $\widehat{\zeta}$ and $\Omega\widehat{\zeta}$ has a right homotopy inverse, there is a choice of extension $\gamma$ such that the diagram in the statement of part~(ii) homotopy commutes. Furthermore, the construction in~\cite{bt2} satisfies a naturality property described in~\cite{t20}*{Remark 2.7}. This ensures that the null homotopies for $\alpha_{N}$ and $\alpha'$ - both of which are connecting maps of fibrations over $BS^{2k-1}$ - are compatible. Moreover the choice of extension defining $\gamma$ that satisfies the diagram as stated in part~(ii) is compatible with a corresponding extension of $\beta$, so that~(\ref{dgm:betagamma}) still homotopy commutes. 
\end{proof}

\section{The Homotopy Type of \(\widehat{E}\)}\label{sec:E-hat_htpytype}

This section proves a splitting for \(\widehat{E}\). To begin we record connectivity and dimension information.

\begin{lem}\label{lem:dim+conn}
    Let \(M\) be \((2k-1)\)-connected with \(\overline{M}\) not contractible. Then
    \begin{enumerate}
        \item[(i)] \(E\) is \((n+2k-1)\)-dimensional, and is at least \((2k-1)\)-connected;
        \item[(ii)] \(\widehat{E}\) is at most \((n+2k-3)\)-dimensional, and is \((2k-1)\)-connected.
    \end{enumerate}
\end{lem}

\begin{proof}
    \textit{Part (i):} As noted earlier, the fact that \(S^{2k-1}\xrightarrow{\alpha} E\xrightarrow{g} E\) is a homotopy fibration of Poinca\'e Duality complexes implies that $E$ is \((n+2k-1)\)-dimensional. Since \(\alpha\simeq\ast\), there is a homotopy equivalence \(\Omega M\simeq S^{2k-1}\times\Omega E\), and therefore \(\pi_{m}(M)\cong\pi_{m-1}(S^{2k-1})\oplus\pi_{m}(E)\) for all $m\geq 1$. In particular, as $M$ is $(2k-1)$-connected this implies that \(\pi_{m}(E)\cong0\) for all \(m\leq2k-1\), so $E$ is at least $(2k-1)$-connected.
    
    \textit{Part (ii):} By hypothesis $M$ is at least simply-connected, so Poincar\'{e} Duality implies that $\overline{M}$ has dimension at most \(n-2\). The homotopy fibration \(S^{2k-1} \xrightarrow{\widehat{\alpha}} \widehat{E} \rightarrow \overline{M}\) then implies that $\widehat{E}$ is at most $(n+2k-3)$-dimensional. By Lemma~\ref{lem:Ehatproperty2}, $\widehat{\alpha}$ is null homotopic, so there is a homotopy equivalence $\Omega\overline{M}\simeq S^{2k-1}\times\Omega\widehat{E}$. 
    Arguing as in Part (i) shows that $\widehat{E}$ is also $(2k-1)$-connected.  
\end{proof} 

The diagram in the statement of Proposition~\ref{prop:betagamma} is an 
iterated homotopy pushout. This implies that the outer rectangle 
\begin{equation*}
    \begin{tikzcd}[row sep=3em, column sep=3em]
        S^{n-1}\amsrtimes S^{2k-1} \arrow[r] \arrow[d, "\gamma"] & \ast \arrow[d] \\ 
        \widehat{E} \arrow[r] & E 
    \end{tikzcd}
\end{equation*}  
is itself a homotopy pushout. Thus there is a homotopy cofibration 
\begin{equation} 
  \label{Ecofib} 
  S^{n-1}\amsrtimes S^{2k-1}\xrightarrow{\gamma} \widehat{E}\rightarrow E. 
\end{equation} 
This is now refined using the inclusion \(i:S^{n-1} \rightarrow S^{n-1}\amsrtimes S^{2k-1}\).

\begin{lem}\label{lem:Ehatproperty4} 
    There is a homotopy cofibration diagram 
    \begin{equation*}  
        \begin{tikzcd}[row sep=3em, column sep = 3em]
                S^{n-1} \arrow[d, "i"] \arrow[r, equal] & S^{n-1} \arrow[d, "\gamma\circ i"] & \\
                S^{n-1}\amsrtimes S^{2k-1} \arrow[r, "\gamma"] \arrow[d, "q"] & \widehat{E} \arrow[r] \arrow[d] & E \arrow[d, equal] \\
                S^{n+2k-2} \arrow[r, "f_E"] & \overline{E} \arrow[r] & E.
        \end{tikzcd}     
    \end{equation*}  
    where $f_E$ is the attaching map for the top cell of $E$. 
\end{lem} 

\begin{proof} 
    Let $X$ be the homotopy cofibre of \(S^{n-1}\stackrel{\gamma\circ i}{\longrightarrow} \widehat{E}\). Then there is a homotopy cofibration diagram
    \begin{equation}\label{dgm:XisEbar}  
        \begin{tikzcd}[row sep=3em, column sep = 3em]
                S^{n-1} \arrow[d, "i"] \arrow[r, equal] & S^{n-1} \arrow[d, "\gamma\circ i"] & \\
                S^{n-1}\amsrtimes S^{2k-1} \arrow[r, "\gamma"] \arrow[d, "q"] & \widehat{E} \arrow[r] \arrow[d] & E \arrow[d, equal] \\
                S^{n+2k-2} \arrow[r, "f_\gamma"] & X \arrow[r] & E
        \end{tikzcd}     
    \end{equation}
    as in the statement of the lemma but with $\overline{E}$ replaced by $X$ and \(f_E\) replaced by the induced map of homotopy cofibres, denoted \(f_\gamma\). By Lemma \ref{lem:dim+conn}(ii), $\widehat{E}$ has dimension at most $n+2k-3$. There are two cases. 
    
    First, if $k>1$ then $n+2k-3>n-1$, so the homology exact sequence associated to the homotopy cofibration from the middle column of (\ref{dgm:XisEbar}) implies that the dimension of $X$ is also at most $n+2k-3$. But then as $E$ has dimension $n+2k-1$, the map $f_{\gamma}$ must attach the \((n+2k-1)\)-cell to $E$. This has two implications: first, that the map \(X \rightarrow E\) in (\ref{dgm:XisEbar}) factors through the $(n+2k-3)$-skeleton $\overline{E}$ of~$E$, resulting in a map \(\epsilon:X\rightarrow\overline{E}\) that induces an isomorphism in homology and is therefore a homotopy equivalence by Whitehead's Theorem, and second, that the composite \(\epsilon\circ f_\gamma\) is homotopic to the attaching map~\(f_E\). Therefore in this case we obtain the asserted homotopy cofibration diagram. 
    
    If $k=1$ then $n+2k-3=n-1$, so $\widehat{E}$ has dimension at most $n-1$. Thus it may be possible that $\gamma\circ i$ attaches an $n$-cell to $\widehat{E}$ in such a way that $H_{n}(X)$ is non-trivial, in which case $X$ is $n$-dimensional. As $E$ is a Poincar\'{e} Duality complex of dimension $n+2k-1=n+1$, $H_{n+1}(E) \cong \mathbb{Z}$. Therefore as $X$ is $n$-dimensional it must be the case that $(f_{\gamma})_{\ast}=0$. This implies that $H_{n}(X)$ injects into $H_{n}(E)$, but as $E$ is simply-connected, Poincar\'{e} duality implies that $H_{n}(E)\cong 0$, and therefore $H_{n}(X)\cong 0$, a contradiction. Thus $\gamma\circ i$ cannot attach an $n$-cell to $X$ that produces nontrivial degree $n$ homology, so $X$ is at most $(n-1)$-dimensional. Now argue as in the previous case to show there exists a homotopy equivalence \(\epsilon:X\rightarrow\overline{E}\) with $\epsilon\circ f_{\gamma}\simeq f_{E}$.
\end{proof}

We close this section with a splitting result which gives the homotopy type of \(\widehat{E}\). This requires a preparatory lemma.

\begin{lem} \label{lem:Ehatproperty1} 
   The skeletal inclusion \(\iota:\overline{E} \rightarrow E\) lifts to a map \(\kappa:\overline{E} \rightarrow \widehat{E}\) with the property that the composite \(\overline{E} \xrightarrow{\kappa} \widehat{E} \xrightarrow{\widehat{g}} \overline{M}\) is homotopic to $\overline{g}$.
\end{lem} 

\begin{proof} 
Consider the diagram 
\begin{equation}
    \begin{tikzcd}[row sep=1.5em, column sep = 1.5em]
        && \overline{E} \arrow[dddr, bend right=20, swap, "\iota"] \arrow[dr, dashed, "\kappa"] \arrow[drrr, bend left=20, "\overline{g}"] &&&& \\
        &&& \widehat{E} \arrow[rr, "\widehat{g}"] \arrow[dd] && \overline{M} \arrow[dd] \\
        &&&&&&& \\
        &&& E \arrow[rr, "g"] && M.
    \end{tikzcd}
\end{equation} 
The square is obtained from juxtaposing the two righthand squares in~(\ref{dgm:null}), implying that it is a homotopy pullback. The homotopy commutativity of the square in the statement of Proposition~\ref{prop:skeltoskel} then implies that there is a pullback map $\kappa$ that makes the two triangular regions homotopy commute. The left triangular region implies that $\kappa$ lifts the skeletal inclusion \(\iota\) while the right triangular region implies that $\widehat{g}\circ\kappa\simeq\overline{g}$. 
\end{proof} 

\begin{prop}\label{prop:Ehatproperty3}  
    If $\overline{M}$ is not contractible then the homotopy cofibration \(S^{n-1}\xrightarrow{\gamma\circ i}\widehat{E}\rightarrow\overline{E}\) of Lemma~\ref{lem:Ehatproperty4} has a homotopy section \(\kappa:\overline{E}\rightarrow\widehat{E}\). Consequently, there is a homotopy equivalence 
    \[
        e=(\gamma\circ i)\perp\kappa:S^{n-1}\vee\overline{E}\xrightarrow{\simeq}\widehat{E}.
    \] 
    with the property that there is a homotopy commutative diagram
    \begin{equation}\label{dgm:e_p2}
        \begin{tikzcd}[row sep=3em, column sep = 3em]
                S^{n-1}\vee\overline{E} \arrow[d, "p_2"] \arrow[r, "e"] & \widehat{E} \arrow[d] & \\
                \overline{E} \arrow[r, equal] & \overline{E} 
        \end{tikzcd} 
    \end{equation} 
    where $p_{2}$ is the pinch map to the second wedge summand.
\end{prop} 

\begin{proof} 
By Lemma \ref{lem:Ehatproperty1}, there is a lift 
\begin{equation*}
    \begin{tikzcd}[row sep=3em, column sep = 3em]
        & \widehat{E} \arrow[d] \\
        \overline{E} \arrow[ur, "\kappa"] \arrow[r, "\iota"] & E
    \end{tikzcd}
\end{equation*}
for some map $\kappa$. On the other hand, letting \(h:\widehat{E}\rightarrow\overline{E}\) be the corresponding vertical map from Lemma~\ref{lem:Ehatproperty4}, then the bottom right square in that lemma implies that there is a homotopy commutative diagram 
\begin{equation*}
    \begin{tikzcd}[row sep=3em, column sep = 3em]
        \widehat{E} \arrow[d] \arrow[r, "h"] & \overline{E} \\
        E. \arrow[from=ur, "\iota"] &
    \end{tikzcd}
\end{equation*}

Consider the composite \(\varepsilon:\overline{E} \xrightarrow{\kappa} \widehat{E} \xrightarrow{h} \overline{E}\); we will to show that it is a homotopy equivalence. Juxtaposing the previous two diagrams along the vertical arrow \(\widehat{E}\rightarrow E\) implies that $\iota\circ\varepsilon\simeq \iota$. Since $E$ is a simply-connected Poincar\'{e} Duality complex of dimension $n+2k-1$, Poincar\'{e} Duality implies that $H^{n+2k-2}(E)\cong 0$. The Universal Coefficient Theorem then implies that $H_{n+2k-2}(E)\cong 0$, and therefore that~$\overline{E}$ is at most $(n+2k-3)$-dimensional. Further, $\iota$ therefore induces an isomorphism in homology in degrees up to and including \(n+2k-3\). Since $\iota_{\ast}\circ\varepsilon_{\ast}=\iota_{\ast}$ it must be the case that~$\varepsilon_{\ast}$ is also an isomorphism in degrees up to and including \(n+2k-3\). As $\overline{E}$ has dimension at most $n+2k-3$, this implies that~$\varepsilon_{\ast}$ is an isomorphism in all degrees. As $E$ is at least simply-connected (by Lemma \ref{lem:dim+conn}(i)), so is $\overline{E}$, and therefore \(\varepsilon\) is a homotopy equivalence by Whitehead's Theorem. Thus, in the context of homotopy cofibration of the statement of the Proposition, the map $h$ has $\kappa$ as a homotopy section. The existence of the asserted homotopy equivalence follows immediately, as does the homotopy commutative square (\ref{dgm:e_p2}).
\end{proof}

\section{Properties of the map \(\gamma\) via assumptions on \(H_{2k}(M)\)} 
\label{sec:gamma via H_2k} 

The ultimate goal is to identify conditions under which the space $E_N$ can be identified as a connected sum of $E$ with a gyration. Proposition~\ref{prop:betagamma} relates $E_N$ to $\widehat{E}$ by a homotopy pullback that involves a map 
\[
    \gamma:S^{n-1}\amsrtimes S^{2k-1} \longrightarrow \widehat{E}
\]
and in Proposition~\ref{prop:Ehatproperty3} it was shown that there is a map \(\kappa:\overline{E} \rightarrow \widehat{E}\) with a left homotopy inverse. The purpose of this section is to relate these by proving Proposition~\ref{prop:varthetagamma}, which gives a condition for when the composite  
\[
    S^{n+2k-2}\xlongrightarrow{j} S^{n-1}\amsrtimes S^{2k-1}\xlongrightarrow{\gamma}\widehat{E}
\]
factors as \(S^{n+2k-2}\xrightarrow{f_{E}}\overline{E}\xrightarrow{\kappa}\widehat{E}\) up to a certain self-equivalence of $S^{n-1}\amsrtimes S^{2k-1}$. This factorisation is crucial for the proof of the Main Theorem in the next section. 

By assumption, $2k<n$, so Lemma~\ref{lem:spherehalfsmash} implies that there is 
a co-$H$-map \(j:S^{n+2k-2} \rightarrow S^{n-1}\amsrtimes S^{2k-1}\) 
with a left homotopy inverse (note that the co-$H$-property will not be needed until Proposition~\ref{prop:varthetagamma}). Consider the composite 
\[
    \vartheta:S^{n+2k-2}\xrightarrow{j} S^{n-1}\amsrtimes S^{2k-1} \xrightarrow{\gamma}\widehat{E} 
     \xrightarrow{e^{-1}}S^{n-1}\vee\overline{E} 
\] 
where $e$ is the homotopy equivalence in Proposition~\ref{prop:Ehatproperty3}. By~\cite{ganea65}, there is 
a homotopy fibration 
\[
    \Omega S^{n-1}\ast\Omega\overline{E} \longrightarrow S^{n-1}\vee\overline{E} \longrightarrow S^{n-1}\times\overline{E}
\] 
where the right map is the inclusion of the wedge into the product, and this homotopy fibration splits after looping to give a homotopy equivalence $\Omega(S^{n-1}\vee\overline{E})\simeq\Omega S^{n-1}\times\Omega\overline{E}\times\Omega(\Omega S^{n-1}\ast\Omega\overline{E})$. Thus when considering \(\vartheta\) as a homotopy class in \(\pi_{n+2k-2}(S^{n-1}\vee\overline{E})\) this splitting gives  $\vartheta\simeq\vartheta_{1}+\vartheta_{2}+\vartheta_{3}$ where   
\begin{align*} 
    \vartheta_{1}&: S^{n+2k-2}\xrightarrow{\vartheta} S^{n-1}\vee\overline{E} \xrightarrow{p_{1}} S^{n-1}\xrightarrow{i_{1}} S^{n-1}\vee\overline{E} \\ 
    \vartheta_{2}&: S^{n+2k-2}\xrightarrow{\vartheta} S^{n-1}\vee\overline{E} \xrightarrow{p_{2}} \overline{E}\xrightarrow{i_{2}} S^{n-1}\vee\overline{E} \\ 
    \vartheta_{3}&: S^{n+2k-2}\xrightarrow{\theta} \Omega S^{n-1}\ast\Omega \overline{E}\rightarrow S^{n-1}\vee\overline{E}
\end{align*} 
for $p_{1}$ and $p_{2}$ the pinch maps to the first and second wedge summands respectively, 
$i_{1}$ and~$i_{2}$ the inclusions of the first and second wedge summands respectively, and some map $\theta$. 

The maps \(\vartheta_1\) and \(\vartheta_2\) can be identified straightaway. Let $\pi_{i}^{S}$ denote the $i^{th}$-stable homotopy group of spheres.  

\begin{lem} 
    \label{lem:vartheta_decomp1} 
    There are homotopies: 
    \begin{itemize} 
        \item[(i)] $\vartheta_{1}\simeq i_{1}\circ\delta$ where $\delta$ represents a torsion class in \(\pi^S_{2k-1}\); 
        \item[(ii)] $\vartheta_{2}\simeq i_{2}\circ f_{E}$. 
    \end{itemize}  
    Thus $\vartheta\simeq i_{1}\circ\delta+i_{2}\circ f_{E}+\vartheta_{3}$. 
\end{lem} 

\begin{proof} 
    \textit{Part (i):} Observe that \(\delta=p_1\circ\vartheta:S^{n+2k-2} \rightarrow S^{n-1}\) represents a class in $\pi_{n+2k-2}(S^{n-1})$. Furthermore, \(\delta\) must be a torsion class, for if $n-1$ is odd then $\pi_{m}(S^{n-1})$ is torsion for all $m>n-1$, while if $n-1$ is even then $\pi_{m}(S^{n-1})$ is torsion for all $m>n-1$ except for $m=2n-3$. In the latter case $n+2k-2=2n-3$ implies $n=2k+1$. But this cannot occur since $M$ is $(2k-1)$-connected with $H_{2k}(M)$ nontrivial, so Poincar\'{e} Duality implies that $n\geq 4k$. This inequality also implies that \(\pi_{n+2k-2}(S^{n-1})\) is isomorphic to the stable homotopy group \(\pi^S_{2k-1}\), which completes the proof of (i). 

    \textit{Part (ii):} Consider the diagram 
    \begin{equation}
        \begin{tikzcd}[row sep=3em, column sep = 3em]
                S^{n+2k-2} \arrow[r, "j"] \arrow[dr, equal] & S^{n-1} \amsrtimes S^{2k-1} \arrow[d, "q"] \arrow[r, "\gamma"] & \widehat{E} \arrow[d] \arrow[r, "e^{-1}"] & S^{n-1}\vee\overline{E} \arrow[d, "p_2"] \\
                & S^{n+2k-2} \arrow[r, "f_E"] & \overline{E} \arrow[r, equal] & \overline{E}.
        \end{tikzcd} 
    \end{equation} 
    The left triangle homotopy commutes since $q$ is a left homotopy inverse for $j$, the middle square homotopy commutes by Lemma \ref{lem:Ehatproperty4}, and the right square homotopy commutes by~(\ref{dgm:e_p2}). The top row is the definition of $\vartheta$, and thus $p_{2}\circ\vartheta\simeq f_{E}$. By definition, $\vartheta_{2}=i_{2}\circ p_{2}\circ\vartheta$, so we obtain $\vartheta_{2}\simeq i_{2}\circ f_{E}$.
\end{proof} 

The identification of $\vartheta_{3}$ will break into two cases, depending on the rank of $H_{2k}(M)$. By assumption, the homotopy fibration \(S^{2k-1}\xrightarrow{\alpha} E \xrightarrow{g} M\) has $\alpha$ null homotopic, so there is a homotopy equivalence $\Omega M\simeq S^{2k-1}\times\Omega E$, implying in particular that 
\begin{equation} \label{eqn:pi2kEM} 
\pi_{2k}(M)\cong\pi_{2k-1}(S^{2k-1})\oplus\pi_{2k}(E). 
\end{equation} 
By hypothesis, $M$ is $(2k-1)$-connected and, by Lemma~\ref{lem:dim+conn}, $E$ is also $(2k-1)$-connected. Therefore the Hurewicz isomorphism implies that 
\begin{equation} \label{eqn:H2kEM} 
H_{2k}(M)\cong H_{2k-1}(S^{2k-1})\oplus H_{2k}(E). 
\end{equation} 
At this point we make the additional assumption that $H_{2k}(M))\cong\oplus_{i=1}^{r}\mathbb{Z}$ for some $r\geq 1$. There are two cases, $r=1$ and $r>1$.  

\subsection{Identifying \(\vartheta_{3}\): the \(r=1\) case}\label{subsec:theta3_r=1}
\hfill
\vspace{5pt}

\begin{lem} 
    \label{lem:M1gen} 
    Suppose that $H_{2k}(M)\cong\mathbb{Z}$. Then $\vartheta_{3}\simeq\ast$. 
\end{lem} 

\begin{proof} 
    The hypothesis that $H_{2k}(M)\cong\mathbb{Z}$ implies by~(\ref{eqn:H2kEM}) that $H_{2k}(E)\cong 0$. But as $E$ is at least $(2k-1)$-connected by Lemma~\ref{lem:dim+conn}, this implies that $E$ is at least $2k$-connected. As $\overline{E}$ is the $(n+2k-2)$-skeleton of~$E$, it too is at least $2k$-connected. Therefore $\Omega S^{n-1}\ast\Omega\overline{E}$ is at least $(n+2k-2)$-connected. By definition, $\vartheta_{3}$ has domain $S^{n+2k-2}$ and factors through $\Omega S^{n-1} \ast\Omega\overline{E}$. Hence $\vartheta_{3}$ is null homotopic.
\end{proof} 

\begin{prop}\label{prop:M1genid} 
    Suppose that $H_{2k}(M)\cong\mathbb{Z}$. Then, localized away from primes dividing $\pi^S_{2k-1}$, there is a homotopy $\vartheta\simeq i_{2}\circ f_{E}$. 
\end{prop} 

\begin{proof} 
    By Lemma~\ref{lem:vartheta_decomp1}, $\vartheta\simeq i_{1}\circ\delta+ i_{2}\circ f_{E}+\vartheta_{3}$, where $\delta$ represents a torsion class in $\pi_{2k-1}^{S}$. By the localisation hypothesis, $\delta$ is null homotopic, and by Lemma~\ref{lem:M1gen}, $\vartheta_{3}$ is null homotopic. Thus $\vartheta \simeq i_{2}\circ f_{E}$. 
\end{proof} 

\subsection{Identifying \(\vartheta_{3}\): the \(r>1\) case} \label{subsec:theta3_r>1}
\hfill
\vspace{5pt} 

This case is much more involved. By assumption, $H_{2k}(M)\cong\oplus_{i=1}^{r}\mathbb{Z}$ 
with $r>1$, so~(\ref{eqn:H2kEM}) implies that  $H_{2k}(E)\cong\oplus_{i=2}^{r}\mathbb{Z}$. Let \(\{x'_{2},\ldots,x'_{r}\}\) be a basis of \(H_{2k}(E)\). As the skeletal inclusion \(\overline{E} \rightarrow E\) induces an isomorphism on $H_{2k}$, we will also use $\{x'_{2},\ldots,x'_{r}\}$ to denote a basis of $H_{2k}(\overline{E})$. As $\overline{E}$ is $(2k-1)$-connected, the Hurewicz isomorphism implies that $\pi_{2k}(\overline{E})\cong\oplus_{i=2}^{r}\mathbb{Z}$ and, for $2\leq i\leq r$, there are maps 
\[
    s'_{i}:S^{2k}\rightarrow\overline{E}
\] 
such that $(s'_{i})_{\ast}$ has image $x'_{i}$. 

By definition, the map $\vartheta_{3}$ factored as a composite \(S^{n+2k-2} \xrightarrow{\theta} \Omega S^{n-1}\ast\Omega\overline{E} \rightarrow S^{n-1}\vee\overline{E}\). Since $\overline{E}$ is $(2k-1)$-connected and 
$\pi_{2k}(\overline{E})\cong\oplus_{i=2}^{r}\mathbb{Z}$, it follows that $\Omega S^{n-1}\ast\Omega\overline{E}$ is $(n+2k-3)$-connected and $\pi_{n+2k-2}(\Omega S^{n-1}\ast\Omega\overline{E})\cong\oplus_{i=2}^{r}\mathbb{Z}$. Thus $\theta$ is homotopic to a linear combination of the $r-1$ maps  that include the bottom cells into $\Omega S^{n-1}\ast\Omega\overline{E}$. By~\cite{ganea65}, the map \(\Omega S^{n-1}\ast\Omega\overline{E} \rightarrow S^{n-1}\vee\overline{E}\)  
is homotopic to the Whitehead product of the composites 
\[
    \Sigma\Omega S^{n-1} \xrightarrow{ev} S^{n-1} \xrightarrow{i_{1}}S^{n-1}\vee\overline{E}\text{\; and \;} \Sigma\Omega\overline{E} \xrightarrow{ev} \overline{E} \xrightarrow{i_{2}} S^{n-1}\vee\overline{E}
\]
where the maps labelled $ev$ are the canonical evaluation maps. Restricting to the inclusions of the bottom cells in $\Omega S^{n-1}\ast\Omega\overline{E}$ therefore gives a sum $\sum_{i=2}^{r} [i_{1},\overline{s}_i]$. Thus  
\[
    \vartheta_{3}=\sum_{i=2}^{r} a_{i}\cdot [i_{1},\overline{s}_{i}] 
\] 
for some coefficients $a_{i}\in\mathbb{Z}$. This lets us rewrite Lemma~\ref{lem:vartheta_decomp1} in this case as follows. 

\begin{lem} 
    \label{lem:vartheta_decomp} 
    Suppose that $H_{2k}(M)\cong\oplus_{i=2}^{r}\mathbb{Z}$ for $r>1$. Then there is a homotopy 
    \[
        \vartheta\simeq i_{1}\circ\delta+i_{2}\circ f_{E}+\sum_{i=2}^{r} a_{i}\cdot [i_{1},\overline{s}_{i}]
    \]  
    for some $a_{i}\in\mathbb{Z}$, where $\delta$ represents a torsion class in \(\pi^S_{2k-1}\).~$\qqed$ 
\end{lem} 

We wish to identify the coefficients $a_{i}$. To gain control we will 
compose to $\overline{M}$.

\begin{lem}\label{lem:Sphere-->M-bar}
    There is a homotopy commutative diagram
    \begin{equation*} 
        \begin{tikzcd}[row sep=3em, column sep = 3em] 
            & S^{n+2k-2} \arrow[dl, swap, "\vartheta"] \arrow[d, "j\circ\gamma"] \\ 
            S^{n-1}\vee\overline{E} \arrow[r, "e"] \arrow[dr, swap, "f_M\perp\overline{g}"] & \widehat{E} \arrow[d, "\widehat{g}"] \\
            & \overline{M}.            
        \end{tikzcd} 
    \end{equation*}     
\end{lem}

\begin{proof} 
    By definition, $\vartheta=e^{-1}\circ\gamma\circ j$, implying that $e\circ\vartheta\simeq\gamma\circ j$, which is written as a homotopy commutative diagram 
    \begin{equation}\label{dgm:varthetalift} 
        \begin{tikzcd}[row sep=3em, column sep = 3em] 
            & S^{n+2k-2} \arrow[dl, swap, "\vartheta"] \arrow[d, "j\circ\gamma"] \\ 
            S^{n-1}\vee\overline{E} \arrow[r, "e"] & \widehat{E}. 
        \end{tikzcd} 
    \end{equation} 
    By Proposition \ref{prop:Ehatproperty3} the homotopy equivalence $e$ is defined as the wedge sum $(\gamma\circ i)\perp\kappa$. By Proposition~\ref{prop:betagamma}, $\widehat{g}\circ\gamma\circ i\simeq f_{M}$ and, by Lemma~\ref{lem:Ehatproperty1}, $\widehat{g}\circ\kappa\simeq\overline{g}$. Therefore 
    \[
        \widehat{g}\circ e\simeq\widehat{g}\circ\left((\gamma\circ i)\perp\kappa\right)\simeq (\widehat{g}\circ\gamma\circ i)\perp (\widehat{g}\circ\kappa)\simeq f_M\perp\overline{g}
    \] 
    so there is a homotopy commutative diagram 
    \begin{equation} \label{dgm:alphadgrm} 
        \begin{tikzcd}[row sep=3em, column sep = 3em] 
            S^{n-1}\vee\overline{E} \arrow[r, "e"] \arrow[dr, swap, "f_M\perp\overline{g}"] & \widehat{E} \arrow[d, "\widehat{g}"] \\
            & \overline{M}.
        \end{tikzcd}
    \end{equation} 
    Juxtaposing (\ref{dgm:varthetalift}) and (\ref{dgm:alphadgrm}) gives the asserted diagram.
\end{proof}

To analyse the diagram in Lemma~\ref{lem:Sphere-->M-bar} we need more detailed information, which starts by relating the bases of $H_{2k}(E)$ and $H_{2k}(M)$ and their duals in cohomology. The map 
\(g:E \rightarrow M\) induces the split injection of $\pi_{2k}(E)$ into $\pi_{2k}(M)$ in~(\ref{eqn:pi2kEM}), and therefore the Hurewicz isomorphims implies that~$g_{\ast}$ induces the split injection of $H_{2k}(E)$ into $H_{2k}(M)$ in~(\ref{eqn:H2kEM}). The cokernel of $g_{\ast}$ in degree $2k$ is therefore isomorphic to $\Z$. If $x'_{1}$ is a generator for this cokernel then $\{x'_{1},g_{\ast}(x'_{2}),\ldots,g_{\ast}(x'_{r})\}$ is a basis for $H_{2k}(M)$. Ambiguously, write $x'_{i}$ for $\overline{g}_{\ast}(x'_{i})$, so that a basis for $H_{2k}(M)$ is $\{x'_{1},\ldots,x'_{r}\}$. As the skeletal inclusion \(\overline{M} \rightarrow M\) induces an isomorphism on $H_{2k}$, we will also use 
$\{x'_{1}, x'_{2},\ldots, x'_{r}\}$ to denote a basis of $H_{2k}(\overline{M})$. For $2\leq i\leq r$, define $s_{i}$ by the composite 
\[
    s_{i}:S^{2k} \xlongrightarrow{s'_{i}} \overline{E} \xlongrightarrow{\overline{g}} \overline{M}.
\] 
Then $(s_{i})_{\ast}$ has Hurewicz image $x'_{i}$. The Hurewicz isomorphism implies that there is also a map 
\[
    s_{1}:S^{2k}\longrightarrow\overline{M}
\] 
such that $(s_{i})_{\ast}$ has Hurewicz image $x'_{1}$. 

Dually, by the Universal Coefficient Theorem, $H^{2k}(\overline{E})\cong H_{2k}(\overline{E})$. For $2\leq i\leq r$, let $x_{i}$ be the cohomological dual of $x'_{i}$. Then $\{x_{2},\ldots,x_{r}\}$ is a basis for $H^{2k}(\overline{E})$. The Universal Coefficient Theorem also implies that $H^{2k}(\overline{M}) \cong H_{2k}(\overline{M})$. For $1\leq i\leq r$, let $x_{i}$ be the cohomological dual of $x'_{i}$. Then $\{x_{1},\ldots,x_{r}\}$ is a basis for $H^{2k}(\overline{M})$. Note that this is compatible with $\overline{g}$ in the sense that $(\overline{g})^{\ast}(x_{i})=x_{i}$ for $2\leq i\leq r$, and as $x'_{1}$ was in the cokernel of $\overline{g}_{\ast}$, we have $x_{1}$ in the kernel of $(\overline{g})^{\ast}$. Succinctly, in degree $2k$ cohomology, $(\overline{g})^{\ast}$ projects $\{x_{1},\ldots,x_{r}\}$ to $\{x_{2},\ldots,x_{r}\}$. Finally, as the skeletal inclusion \(\overline{M} \rightarrow M\) induces an isomorphism on cohomology in degree \(2k\),  ambiguously let $x_{1}\in H^{2k}(M)$ be the class corresponding to $x_{1}\in H^{2k}(\overline{M})$, the context making clear which is meant.

Recall that the sphere $S^{2k-1}$ is rationally an Eilenberg-MacLane space $K(\mathbb{Q},2k-1)$, so it has a rational classifying space $BS^{2k-1}$ whose bottom cell is in dimension \(2k\).   

\begin{prop} 
    \label{rationalbundle} 
    Rationally, the homotopy fibration \(S^{2k-1} \xlongrightarrow{\alpha} E \xlongrightarrow{g} M\) is principal; it is induced by a map \(\chi:M \rightarrow BS^{2k-1}\) that represents the cohomology class \(x_1\). In particular, \(\Omega\chi\) has a right homotopy inverse. 
\end{prop} 

\begin{proof} 
    Start with integral homology and cohomology. First, for skeletal reasons, in degree $2k$ cohomology the map \(g:E \rightarrow M\) behaves exactly as does the map \(\overline{g}: \overline{E} \rightarrow \overline{M}\). So as \(x_{1}\in H^{2k}(\overline{M})\) is in the kernel of \((\overline{g})^{\ast}\), the corresponding element \(x_{1}\in H^{2k}(M)\) is in the kernel of  \(g^\ast\). 
    Second, the element \(x_{1}\in H^{2k}(M)\) is represented by a map \(\chi:M \rightarrow K(\mathbb{Z},2k)\). Its homology dual $x'_{1}\in H_{2k}(M)$ corresponds to $x'_{1}\in H_{2k}(\overline{M})$, so it is the Hurewicz image of the composite 
    \[
        h:S^{2k} \xlongrightarrow{s_{1}} \overline{M} \longrightarrow M.
    \] 
    Consequently, the composite  \(\chi\circ h\) induces an isomorphism in degree~$2k$ homology. 
    
    Now rationalize spaces and maps and take rational homology and cohomology. The element \(x_{1}\) is now represented by a map \(\chi:M \rightarrow BS^{2k-1}\) and the composite 
    \[
        S^{2k} \xlongrightarrow{h} M \xlongrightarrow{\chi} BS^{2k-1}
    \] 
    induces an isomorphism in degree $2k$ homology. Thus, letting \(\widetilde{h}\) denote the adjoint of \(h\), the composite 
    \[
        S^{2k-1} \xlongrightarrow{\widetilde{h}} \Omega M \xlongrightarrow{\Omega\chi} S^{2k-1}
    \] 
    is a degree 1 map in homology and so is a homotopy equivalence. Therefore \(\Omega\chi\) has a right homotopy inverse. 
    
    The map $\chi$ induces a (rational) homotopy fibration sequence 
    \[
        S^{2k-1} \longrightarrow E' \longrightarrow M \xlongrightarrow{\chi} BS^{2k-1}
    \] 
    that defines the space \(E'\). Since \(x_{1}\) is in the kernel of the map \(H^{\ast}(M;\Q) \rightarrow H^{\ast}(E;\Q)\) induced by \(g\), and since \(\chi\) represents \(x_{1}\), the composite \(\chi\circ g\) is null homotopic. Therefore there is a lift 
    \[
        \begin{tikzcd}[row sep=3em, column sep = 3em] 
            & E \arrow[dl, swap, dashed, "\lambda"] \arrow[d, "g"] \\ 
            E' \arrow[r] & M \arrow[r, "\chi"] & BS^{2k-1} 
        \end{tikzcd} 
    \]
    for some map \(\lambda\). Turning this lift into a square and taking homotopy fibres gives a homotopy fibration diagram 
    \begin{equation}\label{EE'dgrm}
        \begin{tikzcd}[row sep=3em, column sep = 3em]
                \Omega M \arrow[d, equal] \arrow[r] & S^{2k-1} \arrow[d, "\lambda'"] \arrow[r] & E \arrow[d, "\lambda"] \arrow[r, "g"] & M \arrow[d, equal] \\
                \Omega M \arrow[r, "\Omega\chi"] & S^{2k-1} \arrow[r] & E' \arrow[r] & M.
        \end{tikzcd}
    \end{equation}
    for some induced map \(\lambda'\) of fibres. The left square in~(\ref{EE'dgrm}) implies that the homotopy equivalence \(\Omega\chi\circ\widetilde{h}\) factors as the composite 
    \[
        S^{2k-1} \xlongrightarrow{\widetilde{h}} \Omega M \longrightarrow S^{2k-1} \xlongrightarrow{\lambda'} S^{2k-1}.
    \]
    Thus \(\lambda'\) must also be an isomorphism in homology and therefore a homotopy equivalence. Hence the Five-Lemma applied to the homotopy fibration diagram~(\ref{EE'dgrm}) implies that \(\lambda\) induces an isomorphism on homotopy groups and is therefore a homotopy equivalence by Whitehead's Theorem. Hence, up to homotopy equivalences, the homotopy fibration \(S^{2k-1} \xrightarrow{\alpha} E \xrightarrow{g} M\) is induced (rationally) by the map \(\chi:M\rightarrow BS^{2k-1}\) representing the cohomology class \(x_{1}\). 
\end{proof} 

The existence of a rational classifying space $BS^{2k-1}$ and a map inducing the homotopy fibration \(g:E \rightarrow M\) implies, by Proposition~\ref{prop:betagamma}~(ii), the existence of the following homotopy commutative diagram. 

\begin{lem} \label{lem:Sphere-->M-bar_rational} 
    Rationally, there is a homotopy commutative diagram 
    \begin{equation*}
        \begin{tikzcd}[row sep=3em, column sep = 6em]
                S^{n-1}\amsrtimes S^{2k-1} \arrow[d, "\simeq"] \arrow[r, "\gamma"] & \widehat{E} \arrow[d, "\widehat{g}"] \\
                S^{n-1}\vee S^{n+2k-2} \arrow[r, "f_M\perp {[f_M,s_{1}]}"] & \overline{M}.
        \end{tikzcd}
    \end{equation*}
\end{lem} 
\vspace{-0.5cm}~$\qqed$\bigskip 

Lemma \ref{lem:Sphere-->M-bar} and Lemma \ref{lem:Sphere-->M-bar_rational} together now imply the following.

\begin{prop}\label{prop:varthetaWhitehead}
    Rationally, there exists a homotopy commutative diagram
    \begin{equation*} 
        \begin{tikzcd}[row sep=3em, column sep = 3em] 
            & S^{n+2k-2} \arrow[dl, swap, "\vartheta"] \arrow[d, "j\circ\gamma"] \arrow[dd, bend left=45, "{[f_M,s_1]}"] \\ 
            S^{n-1}\vee\overline{E} \arrow[r, "e"] \arrow[dr, swap, "f_M\perp\overline{g}"] & \widehat{E} \arrow[d, "\widehat{g}"] \\
            & \overline{M}.            
        \end{tikzcd}
    \end{equation*} 
    In particular, there is a rational homotopy \([f_M,s_1]\simeq(f_M\perp\overline{g})\circ\vartheta\). \qed
\end{prop}

Now, recall that by Lemma~\ref{lem:vartheta_decomp}, \(\vartheta \simeq \vartheta_{1} + \vartheta_{2} + \vartheta_{3}\), where \(\vartheta_{1}\) represents a torsion class in \(\pi_{n+2k-2}(S^{n-1})\), \(\vartheta_{2} \simeq i_{2}\circ f_{E}\), and \(\vartheta_{3} \simeq \sum_{i=2}^{r} a_{i}\cdot [i_{1},\overline{s}_{i}]\). Rationally, $\vartheta_{1}$ is null homotopic, so we are left with the homotopy 
\[
    \vartheta\simeq i_{2}\circ f_{E} + \sum_{i=2}^{r} a_{i}\cdot [i_{1},\overline{s}_{i}].
\]  
Compose $\vartheta$ with \(f_M\perp \overline{g}\). Since by definition \((f_M\perp\overline{g}) \circ \overline{s}_{i} \simeq \overline{g}\circ s'_{i}\simeq s_i\), we obtain 
\[
    (f_M\perp \overline{g})\circ\vartheta \simeq \overline{g}\circ f_{E}+\sum_{i=2}^{r} a_{i}\cdot [f_M,s_{i}].
\] 
Proposition \ref{prop:varthetaWhitehead} therefore gives the following. 

\begin{lem} 
   \label{lem:varthetaWhiteheadeqn} 
   Rationally, there is a homotopy $[f_M,s_{1}]\simeq \overline{g}\circ f_{E}+\sum_{i=2}^{r} a_{i}\cdot [f_M,s_{i}]$.~$\qqed$  
\end{lem} 

The rest of this Subsection is dedicated to showing that $a_{i}=0$ for $2\leq i\leq r$. To do so we will put some hypotheses on the Poincar\'{e} Duality complex $M$. These may not be necessary and it would be interesting to know if a more general result holds.  
 
\begin{hypothesis} \label{hyp} 
    Let \(M\) be a \((2k-1)\)-connected \(n\)-dimensional Poincar\'{e} Duality complex with \(n>4k\) and \(H_{2k}(M)\cong\oplus_{i=1}^{r}\mathbb{Z}\). Suppose that, for \(2\leq i\leq r\), the basis element \(x_{i}\in H^{2k}(\overline{M};\Q)\) satisfies $x_{i}^{2}=0$. 
\end{hypothesis} 

\begin{rem}\label{rem:hypmaps} 
    The condition that $x_{i}^{2}=0$ rationally has a useful interpretation. In what we discuss here, we rationalise spaces and maps. There is a ring isomorphism $H^{\ast}(BS^{2k-1};\mathbb{Q})\cong\mathbb{Q}[y]$ where $\vert y\vert=2k$. Let \(h:BS^{2k-1} \rightarrow BS^{4k-1}\) represent the cohomology class $y^{2}$. A Serre spectral sequence calculation shows that the homotopy fibre of $h$ has the same rational cohomology as $S^{2k}$, implying that it is homotopy equivalent to $S^{2k}$ (this may be thought of as a rational version of the Hopf fibration). Now, for $2\leq i\leq r$, let \(\chi'_i:\overline{M}\rightarrow BS^{2k-1}\) represent the basis element \(x_{i}\in H^{2k}(\overline{M};\Q)\). The composition $h\circ\chi'$ represents $x_{i}^{2}$. So as $x_{i}^{2}=0$ by hypothesis, $h\circ\chi'$ is null homotopic, implying that there is a lift
    \begin{equation*}
        \begin{tikzcd}[row sep=3em, column sep = 3em] 
        & S^{2k} \arrow[d] \\ 
        \overline{M} \arrow[ur, dashed, "\chi_i"] \arrow[r, "\chi'_i"] & BS^{2k-1}. 
        \end{tikzcd} 
    \end{equation*} 
    for some map $\chi_{i}$. 
\end{rem} 

\begin{rem}\label{rem:hyp}
    The hypothesis that $n>4k$ rules out $(2k-1)$-connected $4k$-dimensional Poincar\'{e} Duality complexes. The hypothesis on rational squares being zero is strong but does hold for interesting classes of Poincar\'{e} Duality complexes. For example, if $M$ is a connected sum of products of spheres this holds, or if $\overline{M}$ is rationally homotopy equivalent to a wedge of spheres then this holds. The latter case occurs if $M$ is a torsion-free $(2k-1)$-connected $(4k+1)$-dimensional Poincar\'{e} Duality complex, or more generally if $M$ is a $(2k-1)$-connected Poincar\'{e} Duality complex of dimension less than or equal to $6k-2$~\cite{stanton_theriault}*{Lemma 7.8}. It also holds if $M$ is a moment-angle complex corresponding to a minimally non-Golod triangulation of a sphere~~\cite{stanton_theriault}*{Lemma 8.7}. 
\end{rem}

To work with Hypothesis \ref{hyp} it is helpful to establish information about integral homology and cohomology. Since $M$ is $(2k-1)$-connected and $n$-dimensional, Poincar\'{e} Duality implies that $H_{m}(M)$ vanishes for \(m\) in the range $n-2k<m<n-1$. Thus $\overline{M}$ is a $(n-2k)$-dimensional \(CW\)-complex. Poincar\'{e} Duality also implies that there are isomorphisms $H^{n-2k+1}(M)\cong 0$ and $H^{n-2k}(M)\cong\oplus_{i=1}^{r}\mathbb{Z}$. The Universal Coefficient Theorem implies that $H^{m}(M)\cong H_{m}(M)/T_{m}\oplus T_{m-1}$ for any $m\geq 1$, where $T_{j}$ is the torsion subgroup of $H_{j}(M)$. As $H^{n-2k+1}(M)\cong 0$, we obtain $T_{n-2k}\cong 0$, so as $H^{n-2k}(M)\cong\oplus_{i=1}^{r}\mathbb{Z}$ we obtain $H_{n-2k}(M)\cong\oplus_{i=1}^{r}\mathbb{Z}$. Thus $H_{n-2k}(\overline{M})\cong\oplus_{i=1}^{r}\mathbb{Z}$. Therefore, if $M_{n-2k-1}$ is the $(n-2k-1)$-skeleton of~$M$, then there is a homotopy cofibration 
\[
    \bigvee_{i=1}^{r} S^{n-2k-1}\xlongrightarrow{g} M_{n-2k-1} \longrightarrow \overline{M}
\] 
where $g$ attaches the $(n-2k)$-cells to $\overline{M}$. This homotopy cofibration has 
a connecting map 
\[
    \delta:\overline{M} \longrightarrow \bigvee_{i=1}^{r} S^{n-2k}
\] 
and there is a homotopy coaction 
\[
    \psi:\overline{M} \longrightarrow \overline{M}\vee\left(\bigvee_{i=1}^{r} S^{n-2k}\right)
\]
with the property that $\psi$ composed with the pinch map to $\overline{M}$ is homotopic to the identity map and~$\psi$ composed with the pinch map to $\bigvee_{i=1}^{r} S^{n-2k}$ is homotopic to $\delta$. 

Consider the basis $\{x_{1},\ldots,x_{r}\}$ of $H^{2k}(M)$. By Poincar\'{e} Duality, there is a class $y_{i}\in H^{n-2k}(M)$ for each $x_{i}$ such that $x_{i}\cup y_{i}=z$, where $z$ generates $H^{n}(M)$. The name $x_{i}$ also refers to the corresponding element in $H^{2k}(\overline{M})$, and we will use $y_{i}$ to 
also refer to the corresponding element in $H^{n-2k}(\overline{M})$. The connecting map induces an isomorphism 
\(\delta^*:H^{2k}(\bigvee_{i=1}^{r} S^{n-2k}) \rightarrow H^{2k}(\overline{M})\). There is a one-to-one correspondence between the changes of basis in $H^{\ast}(\bigvee_{i=1}^{r} S^{n-2k})$ and self-homotopy equivalences of $\bigvee_{i=1}^{r} S^{n-2k}$. So changing $\bigvee_{i=1}^{r} S^{n-2k}$ by a self-homotopy 
equivalence if necessary, we may assume that the restriction of $\delta^{\ast}$ to the cohomology 
of the $\ell^{th}$ wedge summand of $\bigvee_{i=1}^{r} S^{n-2k}$ has image $y_{\ell}$. 

Now rationalize spaces and maps. We fix an $\ell$ in the range $2\leq\ell\leq r$. Let 
\[
    \chi'_{\ell}:\overline{M} \longrightarrow BS^{2k-1}
\] 
be the map that represents $x_{\ell}$. As Hypothesis~\ref{hyp} holds, Remark~\ref{rem:hypmaps} 
implies that $\chi'_{\ell}$ factors as 
\[
    \overline{M}\xlongrightarrow{\chi_{\ell}} {S^{2k}} \longrightarrow BS^{2k-1}
\]  
for a map $\chi_{\ell}$. Let 
\[
    \upsilon_{\ell}:\bigvee_{i=1}^{r} S^{n-2k} \longrightarrow S^{n-2k}
\] 
be the pinch map to the $\ell^{th}$-wedge summand. Define $\psi_{\ell}$ by the composite 
\begin{equation}\label{eq:psi_ell_def}
    \psi_{\ell}:\overline{M} \xlongrightarrow{\psi} \overline{M}\vee\left(\bigvee_{i=1}^{r} S^{n-2k}\right)  
      \xlongrightarrow{\chi_{\ell}\vee\upsilon_{\ell}} S^{2k}\vee S^{n-2k}.
\end{equation} 
Note that as $\psi$ is a coaction, composing $\psi_{\ell}$ with the pinch map to $S^{2k}$ is homotopic to $\chi_{\ell}$ and composing~$\psi_{\ell}$ with the pinch map to $S^{n-2k}$ is homotopic to $\upsilon_{\ell}\circ\delta$. 
      
Taking \([f_M,s_{1}]\) and post-composing with $\psi_{\ell}$, by Lemma~\ref{lem:varthetaWhiteheadeqn} we see (using left distributivity) that
\begin{equation} 
  \label{lem:psi_eqn} 
  \psi_{\ell}\circ [f_M,s_{1}]\simeq 
    \psi_{\ell}\circ \overline{g}\circ f_{E}+\sum_{i=2}^{r} a_{i}\cdot\psi_{\ell}\circ [f_M,s_{i}]. 
\end{equation} 
The terms in~(\ref{lem:psi_eqn}) are considered one at a time in what follows, culminating in Theorem~\ref{thm:varthetaclass}, which identifies the class of \(\vartheta\) under Hypothesis \ref{hyp}.        
\begin{lem} 
   \label{lem:psi_eqn1} 
   Rationally, the map $\psi_{\ell}\circ [f_M,s_{1}]$ is null homotopic. 
\end{lem} 

\begin{proof} 
The naturality of the Whitehead product implies that $\psi_{\ell}\circ [f_M,s_{1}]\simeq [\psi_{\ell}\circ f_M,\psi_{\ell}\circ s_{1}]$. It will be shown that $\psi_{\ell}\circ s_{1}$ is null homotopic, which implies that $[\psi_{\ell}\circ f_M,\psi_{\ell}\circ s_{1}]$ is null homotopic. 

Consider the composite 
\[
    S^{2k} \xlongrightarrow{s_{1}} \overline{M} \xlongrightarrow{\psi_{\ell}} S^{2k}\vee S^{n-2k}.
\]
By Hypothesis~\ref{hyp}, the $(2k-1)$-connected, $n$-dimensional Poincar\'{e} Duality complex $M$ 
satisfies $n>4k$, so $n-2k>2k$. Therefore the Hilton-Milnor Theorem implies that $\psi_{\ell}\circ s_{1}$ 
is determined by its pinch map to $S^{2k}$. But $\psi_{\ell}$ composed with the pinch map to 
$S^{2k}$ is homotopic to \(\chi_{\ell}:\overline{M} \rightarrow S^{2k}\). Thus $\psi_{\ell}\circ s_{1}$ is determined by $\chi_{\ell}\circ s_{1}$. By definition, $\chi^{\ast}_{\ell}$ has image $x_{\ell}$ while, by definition, $s_{1}^{\ast}$ projects the basis $\{x_{1},\ldots,x_{r}\}$ of $H^{2k}(M;\Q)$ to $\{x_{1}\}$. Thus $\chi_{\ell}\circ s_{1}$ induces the zero map in cohomology. As $\chi_{\ell}\circ s_{1}$ is a self-map of $S^{2k}$, its homotopy class is determined by its homology class, so $\chi_{\ell}\circ s_{1}$ is null homotopic. Hence $\psi_{\ell}\circ s_{1}$ is null homotopic. 
\end{proof} 

\begin{lem} 
   \label{lem:psi_eqn2} 
   Rationally, the map $\psi_{\ell}\circ \overline{g}\circ f_{E}$ is null homotopic. 
\end{lem} 

\begin{proof} 
By definition, the map \(\chi_{\ell}:\overline{M} \rightarrow S^{2k}\) factors the map \(\chi'_{\ell}:\overline{M} \rightarrow BS^{2k-1}\) that represents $x_{\ell}$. The skeletal inclusion 
of \(\overline{M}\) into \(M\) induces an isomorphism on $H^{2k}$, so $\chi'_{\ell}$ factors as a composite 
\[
    \overline{M} \longrightarrow M \xlongrightarrow{\chi''_{\ell}} BS^{2k-1}.
\]
Now consider the diagram 
\begin{equation*}
    \begin{tikzcd}[row sep=3em, column sep = 3em]
            S^{n+2k-2} \arrow[r, "f_E"] \arrow[dr, "\ast"] & \overline{E} \arrow[d] \arrow[r, "\overline{g}"] & \overline{M} \arrow[d] \arrow[r, "\psi_\ell"] & S^{2k} \vee S^{n-2k} \arrow[d, "p"] \\
            & E \arrow[r, "f_E"] & M \arrow[r, "\chi''_\ell"] & BS^{2k-1}
    \end{tikzcd} 
\end{equation*}
where $p$ is the composite \(S^{2k}\vee S^{n-2k} \xrightarrow{p_{1}} S^{2k} \rightarrow BS^{2k-1}\), with $p_{1}$ the pinch map to the first wedge summand. The left triangle homotopy commutes since $f_{E}$ is the attaching map for the top cell of $E$. The middle square homotopy commutes by Proposition \ref{prop:skeltoskel}. By the definition of $\psi_{\ell}$ in~(\ref{eq:psi_ell_def}), there is a homotopy $p_{1}\circ\psi_{\ell}\simeq\chi_{\ell}$. Thus $p\circ\psi_{\ell}\simeq\chi''_{\ell}$, implying that the right square homotopy commutes. The homotopy commutativity of the diagram  implies that $p\circ\psi_{\ell}\circ \overline{g}\circ f_{E}$ is null homotopic. Therefore there exists a lift \(\lambda\) of $\psi_{\ell}\circ \overline{g}\circ f_{E}$ to the homotopy fibre $X$ of $p$: 
\begin{equation}\label{dgm:psialphalift} 
    \begin{tikzcd}[row sep=3em, column sep = 3em] 
        & S^{n+2k-2} \arrow[d, "\psi_{\ell}\circ \overline{g}\circ f_{E}"] \arrow[dl, dashed, swap, "\lambda"] \\ 
        X \arrow[r] & S^{2k}\vee S^{n-2k} \arrow[r, "p"] & BS^{2k-1}.
    \end{tikzcd} 
\end{equation}  

We will show that \(\lambda\) is null homotopic, from which the assertion of the Lemma follows. The composite defining $p$ produces the homotopy fibration diagram 
\begin{equation*}
    \begin{tikzcd}[row sep=3em, column sep = 3em]
            S^{n-2k}\amsrtimes\Omega S^{2k} \arrow[d, equal] \arrow[r] & X \arrow[d] \arrow[r] & S^{4k-1} \arrow[d] \\
            S^{n-2k}\amsrtimes\Omega S^{2k} \arrow[r] & S^{2k}\vee S^{n-2k} \arrow[r, "p_1"] \arrow[d, "p"] & S^{2k} \arrow[d] \\
            & BS^{2k-1} \arrow[r, equal] & BS^{2k-1}.
    \end{tikzcd} 
\end{equation*}
Here, the homotopy fibre of $p_{1}$ is identified as $S^{n-2k}\amsrtimes\Omega S^{2k}$ by~\cite{ganea65}. Observe that as $p_{1}$ has a right inverse, there is a pullback map \(S^{4k-1} \rightarrow X\) that is a section for the map \(X \rightarrow S^{4k-1}\). The existence of this section implies that
\[
    \pi_{n+2k-2}(X)\cong\pi_{n+2k-2}(S^{4k-1})\oplus\pi_{n+2k-1}(S^{n-2k}\amsrtimes\Omega S^{2k}).
\] 
Thus to show $\lambda$ is null homotopic, it suffices to show that \(\pi_{n+2k-2}(X)\) is (rationally) trivial. 

The space $S^{4k-1}$ has a single rational homotopy group in dimension $4k-1$. If $n+2k-2=4k-1$ then $n=2k+1$, but $n>4k$ by Hypothesis~\ref{hyp}, so it must be the case that $\pi_{n+2k-2}(S^{4k-1})\cong 0$. As for \(\pi_{n+2k-1}(S^{n-2k}\amsrtimes\Omega S^{2k})\), note that \(\Omega S^{2k}\) has its bottom cell in dimension \(2k-1\) and the next in dimension \(4k-2\). Since we are considering the homotopy group in dimension \(n+2k-1\), it is therefore enough to work with the halfsmash \(S^{n-2k}\amsrtimes S^{2k-1}\). Recall that $S^{n-2k}\amsrtimes S^{2k-1}\simeq S^{n-2k}\vee S^{n-1}$ and, by~\cite{ganea65}, the inclusion of the wedge into the product results in a homotopy fibration 
\[
    \Omega S^{n-2k}\ast\Omega S^{n-1} \longrightarrow S^{n-2k}\vee S^{n-1} \longrightarrow S^{n-2k}\times S^{n-1}
\]
that splits after looping. Note that $\Omega S^{n-2k}\ast\Omega S^{n-1}$ is $(2n-2k-3)$-connected. 
As $n>4k$ by Hypothesis~\ref{hyp}, we obtain $n+2k-2<2n-2k-2$. Therefore 
\[
    \pi_{n+2k-2}(S^{n-2k}\vee S^{n-1})\cong\pi_{n+2k-2}(S^{n-2k}\times S^{n-1})\cong\pi_{n+2k-2}(S^{n-2k})\times \pi_{n+2k-2}(S^{n-1}).
\] 
We have $\pi_{n+2k-2}(S^{n-2k})\cong 0$ unless $n-2k$ is even and $n+2k-2=2(n-2k)-1$, but in this case we obtain $n-2k=4k-1$, implying that $n-2k$ is both even and odd, a contradiction. Thus $\pi_{n+2k-2}(S^{n-2k})\cong 0$. Similarly, $\pi_{n+2k-2}(S^{n-1})\cong 0$ unless $n-1$ is even and $n+2k-2=2n-3$, but then $n=2k+1$ while Hypothesis~\ref{hyp} states that $n>4k$. Thus $\pi_{n+2k-2}(S^{n-1})\cong 0$. Hence $\pi_{n+2k-2}(X)\cong 0$, as required.
\end{proof} 

\begin{lem} \label{lem:psi_eqn3} 
    Rationally, if $i\neq\ell$, the map $\psi_{\ell}\circ [f_M,s_{i}]$ is null homotopic. 
\end{lem} 

\begin{proof} 
The naturality of the Whitehead product implies that $\psi_{\ell}\circ [f_M,s_{i}]\simeq [\psi_{\ell}\circ f_M,\psi_{\ell}\circ s_{i}]$. Arguing exactly as in the proof of Lemma~\ref{lem:psi_eqn1} with $s_{1}$ replaced by $s_{i}$ shows that $\psi_{\ell}\circ s_{i}$ is null homotopic. Therefore $[\psi_{\ell}\circ f_M,\psi_{\ell}\circ s_{i}]$ is null homotopic. 
\end{proof} 

\begin{lem} \label{lem:ellWhitehead} 
    Rationally, the map $\psi_{\ell}\circ [f_M,s_{\ell}]$ is non-trivial. 
\end{lem} 

\begin{proof} 
The naturality of the Whitehead product implies that $\psi_{\ell}\circ [f_M,s_{\ell}]\simeq [\psi_{\ell}\circ f_M,\psi_{\ell}\circ s_{\ell}]$. We begin by considering the composite 
\[
    S^{n-1} \xlongrightarrow{f_M} \overline{M} \xlongrightarrow{\psi_{\ell}} S^{2k}\vee S^{n-2k}
\] 
Observe that $\pi_{n-1}(S^{2k})\cong 0$ since $S^{2k}$ has non-trivial rational homotopy groups only in dimensions $2k$ and $4k-1$ while $n>4k$ by Hypothesis~\ref{hyp}. Similarly, $\pi_{n-1}(S^{n-2k})\cong 0$ since $S^{n-2k}$ has non-trivial rational homotopy groups only in dimension $n-2k$ and, if $n$ is even, in dimension $2(n-2k)-1$, both of which cannot occur since $n>4k$. Thus, by the Hilton-Milnor Theorem, $\psi_{\ell}\circ f_M\simeq t\cdot [\iota_{2k},\iota_{n-2k}]$ for some $t\in\mathbb{Q}$, where $\iota_{2k}$ and $\iota_{n-2k}$ are the inclusions of $S^{2k}$ and $S^{n-2k}$ respectively into $S^{2k}\vee S^{n-2k}$. 

Next, we show that $t$ is nonzero. Define the space $D$ by the homotopy cofibration diagram
\begin{equation}\label{dgm:Ddgrm}
    \begin{tikzcd}[row sep=3em, column sep = 3em]
            S^{n-1} \arrow[d, equal] \arrow[r, "f_M"] & \overline{M} \arrow[r] \arrow[d, "\psi_\ell"] & M \arrow[d] \\
            S^{n-1} \arrow[r, "\psi_\ell\circ f_M"] & S^{2k} \vee S^{n-2k} \arrow[r] & D.
    \end{tikzcd} 
\end{equation}
By definition of $\psi_{\ell}$, the map $\psi_{\ell}^{\ast}$ sends the generators of $H^{2k}(S^{2k};\Q)$ and $H^{n-2k}(S^{n-2k};\Q)$ to $x_{\ell}$ and~$y_{\ell}$ respectively. These elements satisfy $x_{\ell}\cup y_{\ell}=z$, where $z$ generates $H^{n}(M;\Q)\cong\mathbb{Q}$. Thus the homotopy commutativity of the right square in~(\ref{dgm:Ddgrm}) implies that $H^{\ast}(D;\Q)\cong H^{\ast}(S^{2k}\times S^{n-2k};\Q)$. The cup product in $H^{\ast}(D)$ is detected by the Whitehead product, so $\psi_{\ell}\circ f\simeq t\cdot[\iota_{2k},\iota_{n-2k}]$ where~$t$ is nonzero. 

Now consider the composite 
\[
    S^{2k} \xlongrightarrow{s_{\ell}} \overline{M} \xlongrightarrow{\psi_{\ell}} S^{2k}\vee S^{n-2k}
\]
By Hypothesis \ref{hyp}, $n>4k$, implying that $n-2k>2k$, so by the Hilton-Milnor Theorem 
$\psi_{\ell}\circ s_{\ell}$ is determined by its pinch map to $S^{2k}$. But $\psi_{\ell}$ composed 
with the pinch map to $S^{2k}$ is homotopic to $\chi_{\ell}$, so $\psi_{\ell}\circ s_{\ell}$ is determined 
by $\chi_{\ell}\circ s_{\ell}$. By definition, $\chi_{\ell}^{\ast}$ has image $x_{\ell}$ in degree~$2k$ 
cohomology while $s_{\ell}^{\ast}$ projects the basis $\{x_{1},\ldots,x_{r}\}$ of 
$H^{2k}(\overline{M};\Q)$ to $\{x_{\ell}\}$. Thus $\chi_{\ell}\circ s_{\ell}$ induces a 
degree one map in cohomology, implying it induces a degree one map in homology, and so 
is a homotopy equivalence. Hence $\psi_{\ell}\circ s_{\ell}$ is homotopic to the 
inclusion $\iota_{2k}$. 

Putting this together, we obtain $[\psi_{\ell}\circ f_M,\psi_{\ell}\circ s_{\ell}]\simeq t\cdot[[\iota_{2k},\iota_{n-2k}],\iota_{2k}]$ for some non-zero rational number \(t\). The iterated Whitehead product on the right side of the homotopy is rationally non-trivial in $\pi_{n+2k-2}(S^{2k}\vee S^{n-2k})$ by the Hilton-Milnor Theorem. Hence $\psi_{\ell}\circ [f_M,s_{\ell}]$ is rationally non-trivial, as asserted. 
\end{proof} 

\begin{lem} \label{lem:psi_eqn4} 
    For each $2\leq\ell\leq r$, the coefficient $a_{\ell}$ of~(\ref{lem:psi_eqn}) is zero. 
\end{lem} 

\begin{proof} 
Rationally, by~(\ref{lem:psi_eqn}) there is a homotopy 
\[
    \psi_{\ell}\circ [f_M,s_{1}]\simeq \psi_{\ell}\circ \overline{g}\circ f_{E}+\sum_{i=2}^{r} a_{i}\cdot\psi_{\ell}\circ [f_M, s_{i}].
\] 
By Lemma~\ref{lem:psi_eqn1}, $\psi_{\ell}\circ [f_M,s_{1}]$ is null homotopic, by Lemma~\ref{lem:psi_eqn2}, 
$\psi_{\ell}\circ \overline{g}\circ f_{E}$ is null homotopic, and by Lemma~\ref{lem:psi_eqn3}, each 
$\psi_{\ell}\circ [f_M,\overline{g}\circ s_{i}]$ is null homotopic if $i\neq\ell$. Thus we are left with 
$\ast\simeq a_{\ell}\cdot\psi_{\ell}\circ [f_M,s_{\ell}]$. But by Lemma~\ref{lem:ellWhitehead}, 
$\psi_{\ell}\circ [f_M,s_{\ell}]$ is non-trivial. Thus the only way that $a_{\ell}\cdot\psi_{\ell}\circ [f_M,s_{\ell}]$ can be null homotopic is if $a_{\ell}=0$. 
\end{proof} 

With Lemma~\ref{lem:psi_eqn4} in hand, we can now identify the integral homotopy class of $\vartheta$ in the \(r>1\) case and under the assumptions of Hypothesis \ref{hyp}. 

\begin{thm} \label{thm:varthetaclass} 
    Assume \(r>1\) and Hypothesis~\ref{hyp} holds. Then there is a homotopy $\vartheta\simeq i_{1}\circ\delta + i_{2}\circ f_{E}$. Consequently, there is a homotopy $\vartheta\simeq i_{2}\circ f_{E}$ after localizing away from primes dividing $\pi^S_{2k-1}$.   
\end{thm} 

\begin{proof} 
By Lemma~\ref{lem:vartheta_decomp}, 
\[
    \vartheta\simeq i_{1}\circ\delta + i_{2}\circ f_{E} + \sum_{i=2}^{r} a_{i}\cdot [i_{1},\overline{s}_{i}]
\] 
where $\delta$ represents a torsion class in $\pi^S_{2k-1}$. By Lemma~\ref{lem:psi_eqn4}, each coefficient $a_{i}$ for $2\leq i\leq r$ is equal to zero. Thus $\vartheta\simeq i_{1}\circ\delta + i_{2}\circ f_{E}$ and the local statement immediately follows. 
\end{proof} 

\begin{rem} 
    It would be interesting to know how much of Hypothesis \ref{hyp} is really necessary, or if there is a different argument that results in Theorem~\ref{thm:varthetaclass} holding for a more general class of Poincar\'{e} Duality complexes. Hypothesis~\ref{hyp} had two components; one minor and one major. The minor of the two was the dimension and connectivity relation $n>4k$, which we used in each of Lemmas~\ref{lem:psi_eqn1} through~\ref{lem:ellWhitehead}. The major part is the demand for basis elements to square trivially in cohomology. As noted in Remark \ref{rem:hypmaps}, this provides the factorisation of the representing maps \(\chi_{i}':\overline{M} \rightarrow BS^{2k-1}\) through \(\chi_i:\overline{M} \rightarrow S^{2k}\). This was really used only in Lemma~\ref{lem:ellWhitehead} to ensure that the iterated Whitehead product $[[\iota_{2k},\iota_{n-2k}],\iota_{2k}]$ is nontrivial. Replacing $S^{2k}\vee S^{n-2k}$ with $BS^{2k-1}\vee S^{n-2k}$ and $\iota_{2k}$ with the inclusion $j$ of $S^{2k}$ into $BS^{2k-1}$, the triple Whitehead product $[[j,\iota_{n-2k}],j]$ is null homotopic, and then the argument for Lemma~~\ref{lem:psi_eqn4} breaks down.
\end{rem} 

The point in identifying $\vartheta$ in Theorem~\ref{thm:varthetaclass} is to give information about the map $\gamma$ in Proposition~\ref{prop:betagamma}. 

\begin{prop} \label{prop:varthetagamma} 
    Suppose that \(H_{2k}(M)\cong\oplus_{i=1}^r\Z\) for some $r\geq 1$. If \(r>1\), also assume that Hypothesis~\ref{hyp} holds. 
    Then there is a homotopy pushout 
    \[\begin{tikzcd}[row sep=3em, column sep=2.5em]
            S^{n+2k-2} \arrow[r, "j"] \arrow[d, "f_E"] & S^{n-1}\amsrtimes S^{2k-1}\arrow[r, "t_{\delta}"] & S^{n-1}\amsrtimes S^{2k-1} \arrow[d, "\gamma"] \\
            \overline{E} \arrow[rr, "\kappa"] & &  \widehat{E}.
        \end{tikzcd}\]
\end{prop} 

\begin{proof} 
By definition, $\vartheta$ is the composite 
\[
    S^{n+2k-2} \xlongrightarrow{j} S^{n-1}\amsrtimes S^{2k-1} \xlongrightarrow{\gamma} \widehat{E} \xlongrightarrow{e^{-1}} S^{n-1}\vee\overline{E}.
\] 
By Proposition~\ref{prop:M1genid} if $r=1$ or by Theorem~\ref{thm:varthetaclass} if $r>1$, $\vartheta\simeq i_{1}\circ\delta + i_{2}\circ f_{E}$. Thus 
\[
    e^{-1}\circ\gamma\circ j\simeq i_{1}\circ\delta+i_{2}\circ f_{E}.
\] 
We will rewrite the term $i_{1}\circ\delta$. By definition of the homotopy equivalence $e$ in 
Proposition~\ref{prop:Ehatproperty3}, the restriction of $e^{-1}\circ\gamma$ to $S^{n-1}$ is the identity map, implying that $i_{1}\simeq e^{-1}\circ\gamma\circ i$. As $q$ is a right homotopy inverse for $j$, we have $\delta\simeq\delta\circ q\circ j$. Thus $i_{1}\circ\delta\simeq e^{-1}\circ\gamma\circ i\circ\delta\circ q\circ j$, giving
\[
    e^{-1}\circ\gamma\circ j\simeq (e^{-1}\circ\gamma\circ i\circ\delta\circ q\circ j)+i_{2}\circ f_{E}.
\] 
By assumption, $2k<n$, so by Lemma~\ref{lem:spherehalfsmash} $j$ is a co-$H$-map. Right distributivity holds for co-$H$-maps and left distributivity holds for all maps, so rearranging and distributing we obtain 
\[
    i_{2}\circ f_{E}\simeq (e^{-1}\circ\gamma\circ j)-(e^{-1}\circ\gamma\circ i\circ\delta\circ q\circ j)\simeq e^{-1}\circ\gamma\circ (1-i\circ\delta\circ q)\circ j.
\] 
By definitions~(\ref{def:epsilondelta}) and~(\ref{def:tdelta}), $\epsilon_{\delta}=i\circ\delta\circ q$ and $t_{\delta}=1-\epsilon_{\delta}$. Thus 
\[
    i_{2}\circ f_{E}\simeq e^{-1}\circ\gamma\circ t_{\delta}\circ j.
\] 
Now pre-compose with the homotopy equivalence $e$ and recall that $\kappa=e\circ i_{2}$; we obtain 
\[
    \kappa\circ f_{E}\simeq\gamma\circ t_{\delta}\circ j.
\]
This proves the homotopy commutativity of the diagram in the statement of the proposition. 

To show that that diagram is a homotopy pushout, consider the horizontal cofibres. Since $t_{\delta}$ is a homotopy equivalence, the homotopy cofibre of $t_{\delta}\circ j$ has the same homotopy type as the homotopy cofibre of $j$, which is $S^{n-1}$. Since $\kappa=e\circ i_{2}$ and $e$ is a homotopy equivalence, the homotopy cofibre of $\kappa$ has the same homotopy type as the homotopy cofibre of $i_{1}$, which again is $S^{n-1}$. Thus there is a homotopy cofibration diagram 
\[
    \begin{tikzcd}[row sep=3em, column sep=2.5em] 
        S^{n+2k-2} \arrow[r, "t_{\delta}\circ j"] \arrow[d, "f_E"] & S^{n-1}\amsrtimes S^{2k-1}\arrow[r]\arrow[d, "\gamma"] & S^{n-1}\arrow[d, "r"] \\
        \overline{E} \arrow[r, "\kappa"] &  \widehat{E}\arrow[r] & S^{n-1} 
    \end{tikzcd}
\] 
for some induced map of cofibres $r$. The homotopy commutativity of the right square implies that the degree of $r$ is the same as the degree of $\gamma$ in $H_{n-1}$, which by definition of $\gamma$ is degree $1$. Thus $r$ is homotopic to the identity map, implying that the left square is a homotopy pushout. 
\end{proof}

\section{Proof of the Main Theorem}
\label{sec:results}

Let us briefly recall the set-up. Let \(M\) be a \((2k-1)\)-connected \(n\)-dimensional Poincar\'{e} Duality complex with \(\overline{M}\) not contractible, and suppose that there is a homotopy fibration 
\[
    S^{2k-1}\xlongrightarrow{\alpha} E \longrightarrow M
\] 
where $\alpha$ is null homotopic (either integrally for \(k\in\lbrace1,2,4\rbrace\) or after localising away from $2$). Let \(N\) be a simply-connected \(n\)-dimensional Poincar\'{e} Duality complex and form the connected sum \(M\# N\). Using the natural collapsing map \(p:M\#N\rightarrow M\) there is a homotopy pullback diagram  
\begin{equation*}
    \begin{tikzcd}[row sep=3em, column sep = 3em]
        S^{2k-1} \arrow[d, equal] \arrow[r, "\alpha_N"] & E_N \arrow[r] \arrow[d] & M\#N \arrow[d, "p"] \\
        S^{2k-1} \arrow[r, "\alpha"] & E \arrow[r] & M 
    \end{tikzcd}
\end{equation*} 
that defines the space $E_{N}$ and the map $\alpha_{N}$. We now restate and prove our main result. 

\begin{thm}\label{thm:main}
    Suppose that \(H_{2k}(M)\cong\oplus_{i=1}^r\Z\). If \(r>1\), also assume that Hypothesis \ref{hyp} holds. 
    Then: 
    \begin{itemize} 
       \item[(i)] there is a homotopy equivalence \(E_{N}\simeq E\conn\mathcal{G}^{2k}_{t_{\delta}}(N)\) for some $\delta\in\pi^{S}_{2k-1}$; 
       \item[(ii)] $\alpha_{N}$ is null homotopic; 
       \item[(iii)] there is a homotopy equivalence  
        \(\Omega(M\conn N)\simeq S^{2k-1}\times\Omega(E\conn\mathcal{G}_{t_{\delta}}^{2k}(N))\). 
    \end{itemize} 
    Further, after localising away from all primes that divide the order of $\pi^{S}_{2k-1}$, there are homotopy equivalences $E_{N}\simeq E\conn\mathcal{G}^{2k}_{0}(N)$ and $\Omega(N\#M)\simeq S^{2k-1}\times\Omega(E\#\mathcal{G}_0^{2k}(N))$.
\end{thm}

\begin{proof} 
    The map $\alpha_N$ was shown to be null homotopic in Lemma~\ref{lem:Ehatproperty2}. Consequently, from the homotopy fibration 
    \(S^{2k-1}\xlongrightarrow{\alpha_N} E_N\longrightarrow M\conn N\) 
    we obtain a homotopy equivalence  
    \begin{equation}\label{loopMdecomp1temp} 
    \Omega(M\conn N)\simeq S^{2k-1}\times\Omega E_N. 
    \end{equation} 
    
    Next, consider the diagram
    \begin{equation}\label{dgm:main_rectangletemp}
        \begin{tikzcd}[row sep=3em, column sep=2.5em]
            S^{n+2k-2} \arrow[r, "j"] \arrow[d, "f_E"] & S^{n-1}\amsrtimes S^{2k-1}\arrow[r,"t_{\delta}"] & S^{n-1}\amsrtimes S^{2k-1} \arrow[r, "f_N\amsrtimes 1"] \arrow[d, "\gamma"] & \overline{N}\amsrtimes S^{2k-1} \arrow[d, "\beta"] \\
            \overline{E} \arrow[rr, "\kappa"] & & \widehat{E} \arrow[r] & E_N.
        \end{tikzcd}
    \end{equation} 
    The right-hand square is a homotopy pushout by Proposition \ref{prop:betagamma} and the left-hand square is a homotopy pushout by Proposition~\ref{prop:varthetagamma}. Thus the outer rectangle is a homotopy pushout as well. By Lemma~\ref{lem:gyrationskel}, the composite along the top of (\ref{dgm:main_rectangletemp}) is the attaching map of the top cell of the gyration $\mathcal{G}_{t_{\delta}}^{2k}(N)$. Therefore Lemma \ref{lem:connsumpo} implies that \(E_N\) has the homotopy type of the connected sum \(E\conn\mathcal{G}^{2k}_{t_{\delta}}(N)\). Consequently, the homotopy equivalence~(\ref{loopMdecomp1temp}) refines to $\Omega(M\conn N)\simeq S^{2k-1}\times\Omega(E\conn\mathcal{G}_{t_{\delta}}^{2k}(N)).$ 
    
    Now localise away from primes that divide the order of $\pi^{S}_{2k-1}$. As $\delta\in\pi^{S}_{2k-1}$ it follows that $\delta$ is null homotopic. By their definitions, $t_{\delta}=1-\epsilon_{\delta}$ where $\epsilon_{\delta}= i\circ\delta\circ q$. Thus $\epsilon_{\delta}$ is null homotopic, implying that $t_{\delta}$ is homotopic to the identity map. Therefore the composite along the top of (\ref{dgm:main_rectangletemp}) is the attaching map of the top cell of the gyration $\mathcal{G}_{0}^{2k}(N)$. Arguing as before then shows that $E_{N}\simeq E\conn\mathcal{G}^{2k}_{0}(N)$ and $\Omega(M\conn N)\simeq S^{2k-1}\times\Omega(E\conn\mathcal{G}_{0}^{2k}(N)).$
\end{proof} 

\section{Applications}\label{sec:applications}

This section gives three applications of the Main Theorem. The first two are both situations in which the least dimensional non-trivial homology group of \(M\) is of rank one, and the third involves simply-connected $5$-dimensional Poincar\'{e} Duality complexes. In particular, the third example requires no additional assumptions in order for Hypothesis \ref{hyp} hold. 

\subsection{Pre-Quantum Line Bundles}\label{subsec:prequantum}
\hfill
\vspace{5pt}

Let \((X,\omega)\) be a symplectic manifold over which there is a \(U(1)\)-bundle \(L\) with connection \(\nabla_L\). The bundle \(L\) is called a \textit{pre-quantum line bundle} if \(\omega\) is equal to the curvature 2-from of \(\nabla_L\) (see for example \cite{morava_SymplecticCobordism}). Theorem~\ref{mainthm} will be applied to a concrete instance of a pre-quantum line bundle from \cite{HJSX} (see also \cite{js}*{Theorem 4.1}) when localised away from $2$. Atiyah and Bott \cite{atiyah-bott} constructed a simply-connected 6-manifold \(A\) with a symplectic structure \(\omega_A\) and an associated pre-quantum line bundle 
\begin{equation}\label{eq:preQ1}
    S^1 \longrightarrow L \longrightarrow A.
\end{equation}
The integral cohomology of \(A\) was calculated via Morse Theory:
\[
    H^i(A)\cong\begin{cases}\Z\text{ if }i=0,2,4,6\\ \Z^4\text{ if }i=3\\ 0\text{ else.}\end{cases}
\]
A result of Wall \cite{wall6mfld}*{Theorem 1} implies that there is a diffeomorphism
\[
    A\cong A'\#(S^3\times S^3)^{\#2}
\]
where \(M^{\#k}\) denotes a connected sum of \(k\) copies of some \(M\), and \(A'\) is simply-connected $6$-manifold with 
\[
    H^i(A')\cong\begin{cases}\Z\text{ if }i=0,2,4,6\\ 0\text{ else}\end{cases}
\]
In \cite{HJSX}*{Theorem 6.3} it is shown that there is a homotopy fibration
\begin{equation}\label{eq:preQ2}
    S^1 \longrightarrow E \longrightarrow A'
\end{equation}
in which the integral cohomology of the total space \(E\) is:
\[
    H^i(E)\cong\begin{cases}\Z\text{ if }i=0,7\\ \Z/4\Z\text{ if }i=4\\ 0\text{ else.}\end{cases}
\]
The homotopy fibrations (\ref{eq:preQ1}) and (\ref{eq:preQ2}) combine to give a homotopy fibration diagram
\begin{equation*}
    \begin{tikzcd}[row sep=3em, column sep = 3em]
        S^1 \arrow[d, equal] \arrow[r] & L \arrow[r] \arrow[d] & A \arrow[d, "p"] \\
        S^1 \arrow[r] & E \arrow[r] & A' 
    \end{tikzcd}
\end{equation*}
where the map \(p:A\rightarrow A'\) is the standard collapsing map from the connected sum. Both \(E\) and \(L\) are simply-connected, so Theorem~\ref{mainthm} applies.

\begin{thm}\label{thm:preQ} 
   There is a homotopy equivalence $L\simeq E\#\mathcal{G}^2_{t_{\delta}}((S^3\times S^3)^{\#2})$ 
   for some $\delta\in\pi^{S}_{1}$. Further, localised away from 2, there is a homotopy equivalence
    \[
        L\simeq E\#(S^3\times S^4)^{\#4}.
    \]
\end{thm}

\begin{proof}
    As \(H_2(A')\cong\Z\), \(k=1\) and \(\pi_1^S\cong \Z/2\Z\), by {the Main Theorem} there is a homotopy equivalence
    \[ 
         L\simeq E\#\mathcal{G}^2_{t_{\delta}}((S^3\times S^3)^{\#2})  
     \] 
    and localised away from $2$ there is a homotopy equivalence 
    \[
        L\simeq E\#\mathcal{G}^2_0((S^3\times S^3)^{\#2}). 
    \] 
    The latter can be refined. By \cite{chenery:fico}*{Theorem 4.1} there is a homotopy equivalence \(\mathcal{G}^2_0((S^3\times S^3)^{\#2})\simeq\mathcal{G}^2_0(S^3\times S^3)^{\#2}\), and by \cite{chenery:fico}*{Theorem 5.8} there is a homotopy equivalence \(\mathcal{G}^2_0(S^3\times S^3)\simeq(S^3\times S^4)^{\#2}\). Thus, localised away from $2$, $L\simeq E\#\mathcal{G}^2_0((S^3\times S^3)^{\#2})$.
\end{proof}

\subsection{Rational Projective Spaces}\label{subsec:rationalproj}
\hfill
\vspace{5pt}

In this subsection all spaces and maps are rationalised. In this case, \(S^{2k-1}\) is an Eilenberg-MacLane space, and so has a classifying space \(BS^{2k-1}\) whose rational cohomology is 
\[
    H^*(BS^{2k-1};\Q)\cong\Q[x]\text{\; where \;} |x|=2k.
\] 
There is a canonical filtration of $BS^{2k-1}$ arising from the fact that it is rationally homotopy equivalent to $\Omega S^{2k+1}$. Integrally, there is an algebra isomorphism $H_{\ast}(\Omega S^{2n+1})\cong\mathbb{Z}[y]$ for $\vert y\vert= 2k$. For \mbox{$n\geq 1$}, let $\mathbb{P}_{n}$ be the $2nk$-skeleton of $\Omega S^{2n+1}$, so there is an algebra isomorphism $H_{\ast}(\mathbb{P}_n) \cong \mathbb{Z}[y]/\langle y^{n+1}\rangle$. Rationally, the truncated polynomial algebra $\mathbb{Q}[y]\langle y^{n+1}\rangle$ is self-dual. So if $x$ is the rational dual of $y$, then $H^{\ast}(\mathbb{P}_n;\Q)\cong\mathbb{Q}[x]/\langle x^{n+1}\rangle$. Note that this algebra isomorphism implies that $\mathbb{P}_{n}$ is (rationally) a Poincar\'{e} Duality complex.
  
The map \(\mathbb{P}_n\rightarrow BS^{2k-1}\) representing the generator \(x\) induces a homotopy fibration 
\[
    S^{2k-1} \longrightarrow \mathbb{S}_n \longrightarrow \mathbb{P}_n.
\]
that defines the space \(\mathbb{S}_n\). This generalises the $k=1$ case when \(\mathbb{P}_n\simeq\C P^n\) and \(\mathbb{S}_n\simeq S^{2n+1}\). In general, a Serre spectral sequence reveals that 
\[
    H_\ast(\mathbb{S}_n;\Q)\cong\begin{cases}\Q\text{ if }\ast=2k(n+1)-1\\ 0\text{ else}\end{cases}
\]
which implies that \(\mathbb{S}_n\) is a rational homology sphere. Simply-connected rational homology spheres are rational homotopy spheres, so there is a rational homotopy equivalence \(\mathbb{S}_n \simeq S^{2k(n+1)-1}\). 

Now consider a \(2kn\)-dimensional Poincar\'e Duality complex \(N\) and take the connected sum \(N\#\mathbb{P}_n\). Let $E_{N}$ be the homotopy pullback of \(p:N\conn\mathbb{P}_n \rightarrow \mathbb{P}_n\) and \(S^{2k(n+1)-1} \rightarrow \mathbb{P}_n\). Since \(H_{2k}(\mathbb{P}_n;\Q)\cong\mathbb{Q}\), by the Main Theorem there is a homotopy equivalence $E_{N}\simeq S^{2k(n+1)-1}\conn\mathcal{G}_{0}^{2k}(N)$. In general, taking the connected sum with a sphere does not change the homotopy type, so we obtain a homotopy fibration diagram  
\begin{equation*}
    \begin{tikzcd}[row sep=3em, column sep = 3em]
        S^{2k-1} \arrow[d, equal] \arrow[r] & \mathcal{G}_0^{2k}(N) \arrow[r] \arrow[d] & N\#\mathbb{P}_n \arrow[d, "p"] \\
        S^{2k-1} \arrow[r] & S^{2k(n+1)-1} \arrow[r] & \mathbb{P}_n.
    \end{tikzcd}
\end{equation*} 
By Lemma~\ref{lem:Ehatproperty2}, the map \(S^{2k-1} \rightarrow \mathcal{G}_{0}^{2k}(N)\) 
is null homotopic, implying that there is a homotopy equivalence \(\Omega(N\#\mathbb{P}_n)\simeq S^{2k-1}\times\Omega(\mathcal{G}_0^{2k}(N))\). By~\cite{huangtheriault}, 
\[
    \Omega(\mathcal{G}_0^{2k}(N))\simeq\Omega\overline{N}\times\Omega\Sigma^{2k} H
\] 
where $\overline{N}$ is the $(2kn-1)$-skeleton of $N$ and $H$ is the homotopy fibre of the attaching map \(S^{2kn-1} \rightarrow \overline{N}\) for the top cell of $N$. Thus we obtain a homotopy equivalence 
\[
    \Omega(N\conn\mathbb{P}_n)\simeq S^{2k-1}\times\Omega\overline{N}\times\Omega(\Sigma^{2k} H).
\]

\subsection{5-dimensional Poincar\'{e} Duality complexes}\label{subsec:5dim}
\hfill
\vspace{5pt}

Let \(M\) be a simply-connected 5-dimensional Poincar\'e Duality complex with \(H_2(M)\cong\oplus_{i=1}^{r}\mathbb{Z}\) for some $r\geq 1$. Suppose there exists a homotopy fibration 
\begin{equation}\label{eq:5dim}
    S^1 \xlongrightarrow{\alpha} E \longrightarrow M 
\end{equation}
where $\alpha$ is null homotopic. Since $H^{4}(M;\mathbb{Q})\cong 0$, all squares of degree $2$ 
basis elements in $H^{2}(M;\mathbb{Q})$ are zero, implying that Hypothesis \ref{hyp} holds. 

Apply the Main Theorem with \(N\) some other simply-connected 5-dimensional Poincar\'e Duality complex. As $k=1$ and $\pi_{1}^{S}\cong\mathbb{Z}/2\mathbb{Z}$, there is a homotopy equivalence 
$E_{N}\simeq E\conn\mathcal{G}^{2}_{t_{\delta}}(N)$ for some $\delta\in\pi^{S}_{1}$. After 
localising away from $2$, there is a homotopy fibration 
\[ 
    S^1 \xlongrightarrow{\alpha_N} E \conn \mathcal{G}_0^2(N) \longrightarrow M \conn N   
\] 
where $\alpha_{N}$ is null homotopic and a homotopy equivalence 
$\Omega(M\conn N)\simeq S^{1}\times\Omega(E\conn\mathcal{G}_{0}^{2}(N))$.

\appendix
\section{The Cube Lemma}\label{appendix} 

Here we prove a version of the classical result known as Mather's Cube Lemma (cf. \cite{mather}) that works with juxtapositions of cubes.  As terminology, a pushout is a strict (that is, point-set) 
pushout and similarly for a pullback. These are distinguished from homotopy pushouts 
and homotopy pullbacks. 

\begin{thm} \label{strictcube} 
    Suppose that there is a diagram of iterated pushouts 
    \begin{equation*}
        \begin{tikzcd}[row sep=3em, column sep = 3em]
            A' \arrow[d, "\overline{g}"] \arrow[r, "f'"] & A \arrow[r, "f"] \arrow[d] & B \arrow[d] \\
            C' \arrow[r] & C \arrow[r] & D
        \end{tikzcd}
    \end{equation*}   
    where $f'$, $f$ and $\overline{g}$ are subspace inclusions and suppose that there is a epimorphism \(h:H \rightarrow D\). Then there is a commutative cube 
    \begin{equation*}
        \begin{tikzcd}[row sep=1em, column sep=1em]
            E' \arrow[rr] \arrow[dd] \arrow[dr] && E \arrow[rr] \arrow[dr] \arrow[dd] &&
            F \arrow[dd] \arrow[dr] \\
            & G' \arrow[rr, crossing over] && G \arrow[rr, crossing over] && H \arrow[dd, "h"] \\
            A' \arrow[rr] \arrow[dr] && A \arrow[rr] \arrow[dr] && B \arrow[dr]\\
            & C' \arrow[rr] \arrow[from=uu, crossing over] && C \arrow[rr] \arrow[from=uu, crossing over] && D
        \end{tikzcd}
    \end{equation*}
    where the sides are all pullbacks and the two top faces are pushouts. 
\end{thm} 

\begin{rem} 
    Restricting to the right pushout in the hypotheses and the right cube as output gives Mather's Cube Lemma in the context of strictly commuting diagrams. Mather's Cube Lemma allows for homotopies: it states that if there is a homotopy commutative cube in which the bottom face is a homotopy pushout and the four sides are homotopy pullbacks then the top face is also a homotopy pushout. There is a technicality here. The homotopy commutativity of the cube does not simply mean that each of the six faces in the cube individually homotopy commutes, it also means there is a coherency condition between the homotopies; this is usually phrased now as saying the cube homotopy commutes in the sense of Mather. The coherency condition appears because Mather's statement is strong: given \textit{any} cube that homotopy commutes in the sense of Mather, if the bottom face is a homotopy pushout and the four sides are homotopy pullbacks then the top face is a homotopy pushout. What we are after is a weaker version where a homotopy commutative cube means only that all six sides homotopy commute: given a homotopy pushout and a map to the pushout, there is \textit{at least one instance} of a homotopy commutative cube in which the four sides are homotopy pullbacks and the top face is a homotopy pushout. The strict commutativity in Theorem~\ref{strictcube} is designed to realize an instance of a cube, which can then be used in the homotopy commutative case to produce an instance of a cube without worrying about coherency conditions (as stated in Theorem~\ref{homotopycube}). The construction hasthe added benefit of working with iterated pushouts as stated in Theorem~\ref{strictcube}. 
\end{rem} 

We first prove a preliminary lemma. 

\begin{lem}\label{strictpb} 
    Suppose that there is an iterated pullback diagram  
    \begin{equation*}
        \begin{tikzcd}[row sep=3em, column sep = 3em]
            X_{f\circ g} \arrow[r, "\overline{g}"] \arrow[d] & X_f \arrow[r, "f'"] \arrow[d] & X \arrow[d, "h"] \\
            T \arrow[r, "g"] & S \arrow[r, "f"] & Y
        \end{tikzcd}
    \end{equation*}      
    where $g$ and $f$ are subspace inclusions. Then there are homeomorphisms $X_{f\circ g}\cong h^{-1}(T)$ and $X_{f}\cong h^{-1}(S)$ with the property that $\overline{g}$ identifies with the inclusion of $h^{-1}(T)$ into $h^{-1}(S)$ and~$f'$ identifies with the inclusion of $h^{-1}(S)$ into $X$. 
\end{lem} 

\begin{proof} 
This is straightforward but a proof is included to be complete. By its definition as a pullback, 
\[
    X_{f}=\{(s,x)\in S\times X\mid f(s)=h(x)\}.
\] 
Since $s$ is a subspace inclusion, to simplify notation, identify $f(S)$ with $S$ and write $f(s)=s$. Then 
$X_{f}=\{(s,x)\in S\times X\mid s=h(x)\}=\{(h(x),x)\in S\times X\}$. Define maps \(\alpha: X_{f} \rightarrow h^{-1}(S)\) and \(\beta:h^{-1}(S) \rightarrow X_{f}\) by $\alpha(h(x),x)=x$ and $\beta(x)=(h(x),x)$. Both maps are well defined since $h(x)\in S$, both are continuous since $h$ is, and clearly $\alpha\circ\beta$ and $\beta\circ\alpha$ are the respective identity maps. Thus $X_{f}$ is homeomorphic to $h^{-1}(S)$, and under the homeomorphism $\alpha$ the map $f'$ is the inclusion of $h^{-1}(S)$ into~$X$. 

Next, consider $X_{f\circ g}$. In the iterated pullback defining $X_{f\circ g}$, the outer rectangle 
is also a pullback, so we may regard $X_{f\circ g}$ as being obtained from $X_{f}$ by  
restricting $S\times X$ to $T\times X$. The corresponding restrictions of $\alpha$ and $\beta$ 
then imply there is a homeomorphism $X_{f\circ g}\cong h^{-1}(T)$, and because this is obtained 
by restricting $\alpha$, in particular, the map 
\(\overline{g}:X_{f\circ }g \rightarrow X_{f}\) 
identifies with the inclusion of $h^{-1}(T)$ into $h^{-1}(S)$.  
\end{proof}  

\begin{proof}[Proof of Theorem~\ref{strictcube}]  
First consider the pushout $A$-$B$-$C$-$D$. Note that as $f'$ and $\overline{g}$ are subspace inclusions, so 
is the induced pushout map \(g:A \rightarrow C\). Since $f$ and $g$ are subspace inclusions, we have $D=B\cup_{A} C$ and the maps \(b:B \rightarrow D\) and \(c:C \rightarrow D\) are subspace inclusions. Define the spaces $H_{b}$, $H_{b\circ f}$, $H_{c}$ and $H_{c\circ g}$ by the iterated pullbacks 
\begin{equation}\label{cubepbs} 
    \begin{tikzcd}[row sep=3em, column sep = 3em]
        H_{b\circ f} \arrow[r] \arrow[d] & H_b \arrow[r] \arrow[d] & H \arrow[d, "h"] && H_{c\circ g} \arrow[r] \arrow[d] & H_c \arrow[r] \arrow[d] & H \arrow[d, "h"] \\
        A \arrow[r, "f"] & B \arrow[r, "b"] & D && A \arrow[r, "g"] & C \arrow[r, "c"] & D
    \end{tikzcd}
\end{equation} 
By Lemma~\ref{strictpb}, there are homeomorphisms 
\[
H_{b}\cong h^{-1}(B)\text{, \;}H_{b\circ f}\cong h^{-1}(A)\text{, \;}H_{g}\cong h^{-1}(C)\text{\; and \;}H_{c\circ g}\cong h^{-1}(A)
\] 
and the maps along the top rows 
in the left and right diagrams in~(\ref{cubepbs}) identify with the subspace inclusions 
\(h^{-1}(A) \rightarrow h^{-1}(B) \rightarrow H\) 
and 
\(h^{-1}(A) \rightarrow h^{-1}(C) \rightarrow H\)  
respectively. This has two implications. First, as $b\circ f=c\circ g$ the homeomorphisms for $H_{b\circ f}$ 
and $H_{c\circ g}$ with $h^{-1}(A)$ identify to give $H_{b\circ f}= H_{c\circ g}$. Second, 
since $D=B\cup_{A} C$ and $h$ is an epimorphism, we have $H=h^{-1}(B)\cup_{h^{-1}(A)} h^{-1}(C)$. 
Therefore the compatible homeomorphisms imply that there is a pushout 
\begin{equation*}
    \begin{tikzcd}[row sep=3em, column sep = 3em]
        H_{b\circ f} \arrow[r] \arrow[d] & H_b \arrow[d] \\
        H_c \arrow[r] & H.
    \end{tikzcd}
\end{equation*}  
Finally, using the given pushout $A$-$B$-$C$-$D$ as the bottom face of a cube, the four pullback squares 
in~(\ref{cubepbs}) as the four sides, and the pushout for $H$ as the top face, we obtain 
the right-hand commutative cube in the statement of the theorem. 

Next, consider the pushout $A'$-$A$-$C'$-$C$. Note in~(\ref{cubepbs}) that as $h$ is 
an epimorphism so is the induced map  \(H_{c} \rightarrow C\). Since $f'$ and $\overline{g}$ are subspace inclusion so are the induced pushout maps \(g:A \rightarrow C\) (the same map $g$ from above) and \(c':C' \rightarrow C\). The same argument as above then shows that there are iterated pullbacks 
\begin{equation}\label{cubepbs2} 
    \begin{tikzcd}[row sep=3em, column sep = 3em]
        H_{c\circ g\circ f'} \arrow[r] \arrow[d] & H_{c\circ g} \arrow[r] \arrow[d] & H_c \arrow[d] && H_{c\circ c'\circ \overline{g}} \arrow[r] \arrow[d] & H_{c\circ c'} \arrow[r] \arrow[d] & H_c \arrow[d] \\
        A' \arrow[r, "f'"] & A \arrow[r, "g"] & C && A' \arrow[r, "\overline{g}"] & C' \arrow[r, "c'"] & C
    \end{tikzcd}
\end{equation}
and as $g\circ f'=c'\circ \overline{g}$ we have $H_{c\circ g\circ f'}=H_{c\circ c'\circ \overline{g}}$ and there is a pushout 
\begin{equation*}
    \begin{tikzcd}[row sep=3em, column sep = 3em]
        H_{c\circ g\circ f'} \arrow[r] \arrow[d] & H_{c\circ g} \arrow[d] \\
        H_{c\circ c'} \arrow[r] & H_c.
    \end{tikzcd}
\end{equation*}  
Note that $H_{c\circ g}$ has already been identified $H_{b\circ f}$ and in such a way that their 
maps to $H_{c}$ also identify. Thus using the pushout $A'$-$A$-$C'$-$C$ as the bottom face 
of a cube, the four pullback squares in~(\ref{cubepbs2}) as the four sides, and the pushout for $H_{c}$ 
as the top face, we obtain the left-hand cube in the statment of the lemma, with the additional 
observation that the two cubes can be juxtaposed along the common face 
$H_{b\circ f}$-$H_{c}$-$A$-$C$.  
\end{proof}  

Theorem~\ref{strictcube} is now generalised. 

\begin{thm} \label{homotopycube} 
    Suppose that there is a diagram of iterated homotopy pushouts 
    \begin{equation*}
        \begin{tikzcd}[row sep=3em, column sep = 3em]
            A' \arrow[d, "\overline{g}"] \arrow[r, "f'"] & A \arrow[r, "f"] \arrow[d] & B \arrow[d] \\
            C' \arrow[r] & C \arrow[r] & D
        \end{tikzcd}
    \end{equation*}   
    and suppose that there is a map \(h:H \rightarrow D\). Then there is a homotopy commutative cube 
    \begin{equation*}
        \begin{tikzcd}[row sep=1em, column sep=1em]
            E' \arrow[rr] \arrow[dd] \arrow[dr] && E \arrow[rr] \arrow[dr] \arrow[dd] &&
            F \arrow[dd] \arrow[dr] \\
            & G' \arrow[rr, crossing over] && G \arrow[rr, crossing over] && H \arrow[dd, "h"] \\
            A' \arrow[rr] \arrow[dr] && A \arrow[rr] \arrow[dr] && B \arrow[dr]\\
            & C' \arrow[rr] \arrow[from=uu, crossing over] && C \arrow[rr] \arrow[from=uu, crossing over] && D
        \end{tikzcd}
    \end{equation*}
    where the sides are all homotopy pullbacks and the two top faces are homotopy pushouts. 
\end{thm} 

\begin{proof} 
    Replace $A$ and $C'$ by the mapping cylinders $M_{f'}$ and $M_{\overline{g}}$ of $f'$ and $\overline{g}$ respectively and replace $C$ by the pushout $M$ of \(A' \rightarrow M_{f'}\) and \(A' \rightarrow M_{\overline{g}}\). Replace $B$ by the mapping cylinder $M_{d}$ of the composite  
    \[
        d: M_{f'} \xlongrightarrow{\simeq}{A} \xlongrightarrow{f} {B}
    \]  
    and replace $D$ by the pushout $N$ of \(M_{f'} \rightarrow M_{d}\) and \(M_{f'} \rightarrow M\). Finally, replace $H$ by the mapping path space $P$ of the composite 
    \[
        H \xlongrightarrow{h} D \xlongrightarrow{\simeq} N.
    \]
    Then there is an iterated diagram of (strict) pushouts 
    \begin{equation*}
        \begin{tikzcd}[row sep=3em, column sep = 3em]
            A' \arrow[d] \arrow[r] & M_{f'} \arrow[r] \arrow[d] & M_d \arrow[d] \\
            M_{\overline{g}} \arrow[r] & M \arrow[r] & N
        \end{tikzcd}
    \end{equation*}  
    where each map is a subspace inclusion, there is an epimorphism \(P \rightarrow N\), and this data is equivalent up to homotopy to the given iterated homotopy pushout and map $h$. Now apply Theorem~\ref{strictcube}. 
\end{proof}

\bibliographystyle{amsplain}
\bibliography{bib}

@article{t20,
  title={Homotopy fibrations with a section after looping},
  author={Theriault, S.},
  journal={Mem. Amer. Math. Soc.}, 
  volume={299}, 
  number={1499}, 
  year={2024}
}

@article{bt2,
      title={Homotopy groups of highly connected {P}oincare duality complexes}, 
      author={P. Beben and S. Theriault},
      journal={Documenta Mathematica}, 
      volume={27}, 
      pages={183-211},
      year={2022} 
}

@article{mather,
  title={Pull-backs in homotopy theory},
  author={Mather, M.},
  journal={Canadian Journal of Mathematics},
  volume={28},
  number={2},
  pages={225--263},
  year={1976},
  publisher={Cambridge University Press}
}

@article{wall6mfld,
  title={Classification problems in differential topology. {V}. {O}n certain $6$-manifolds},
  author={Wall, C.T.C.},
  journal={Inventiones Mathematicae},
  volume={1},
  number={4},
  pages={355--374},
  year={1966},
  publisher={Springer}
}

@article{quinn,
  title={Surgery on {P}oincar{\'e} and normal spaces},
  author={Quinn, F.},
  journal={Bulletin of the American Mathematical Society},
  volume={78},
  number={2},
  pages={262--267},
  year={1972}
}

@incollection{js,
  title={Bundles over connected sums},
  author={Jeffrey, L. and Selick, P.},
  booktitle={Toric Topology and Polyhedral Products, {F}ields {I}nstitute {C}ommunications, vol. 89},
  series={Toric Topology and Polyhedral Products},
  volume={89},
  pages={149--156},
  year={2024},
  publisher={Springer}
}

@article{duan,
  title={Circle actions and suspension operations on smooth manifolds},
    DOI={10.1017/S0305004125101849}, 
    journal={Mathematical Proceedings of the Cambridge Philosophical Society}, 
    author={Duan, H.}, 
    year={2026}, 
    pages={1–20, doi:10.1017/S0305004125101849}
}

@article{chenfib,
    title={The rational homotopy type of homotopy fibrations over connected sums},
    author={Chenery, S.},
    volume={66},
    number={1},
    journal={Proceedings of the Edinburgh Mathematical Society}, 
    publisher={Cambridge University Press},
    year={2023},
    pages={133–142}
}

@book{selick,
    AUTHOR = {Selick, P.},
     TITLE = {Introduction to homotopy theory},
    SERIES = {Fields Institute Monographs},
    VOLUME = {9},
 PUBLISHER = {American Mathematical Society, Providence, RI},
      YEAR = {1997}
}

@article{huangtheriault,
  title={Homotopy of manifolds stabilized by projective spaces},
  author={Huang, R. and Theriault, S.},
  journal={Journal of Topology}, 
  volume={16}, 
  pages={1237--1257}, 
  year={2023}
}

@article{basu-ghosh,
  title={Sphere fibrations over highly connected manifolds},
  author={Basu, S. and Ghosh, A.},
  journal={Journal of the London Mathematical Society},
  volume={110}, 
  pages={Paper No. e70002},
  year={2024}
}

@article{klpt,
  title={Stable classification of 4-manifolds with 3-manifold fundamental groups},
  author={Kasprowski, D. and Land, M. and Powell, M. and Teichner, P.},
  journal={Journal of Topology},
  volume={10},
  number={3},
  pages={827--881},
  year={2017}
}

@article{ganea65,
  title={A generalization of the homology and homotopy suspension},
  author={Ganea, T.},
  journal={Commentarii Mathematici Helvetici},
  volume={39},
  pages={295-322},
  year={1965}
}

@article{galaz-garcia--reiser,
  title={Free torus actions and twisted suspensions},
  author={Galaz-Garc{\'\i}a, F. and Reiser, P.},
  journal={Forum Sigma Mathematics}, 
  volume={13}, 
  pages={Paper No. e3, 31pp.}, 
  year={2025}
}

@article{bosio_meersseman,
    author = {Bosio, F. and Meersseman, L.},
    title = {Real quadrics in \(\mathbb{C}^n\), complex manifolds and convex polytopes},
    volume = {197},
    journal = {Acta Mathematica},
    number = {1},
    publisher = {Institut Mittag-Leffler},
    pages = {53 -- 127},
    year = {2006},
    doi = {10.1007/s11511-006-0008-2},
    URL = {https://doi.org/10.1007/s11511-006-0008-2}
}

@article{gonzalezacuna,
    author = {Gonz{\'a}lez Acu{\~n}a, F.},
    title = {Open Books},
    journal = {University of Iowa Notes},
    year = {1975},
    pages = {12 pages}
}

@article{huang_inertness24,
  title={Comparison techniques on inert top cell attachments},
  author={Huang, R.},
  journal={arXiv:2408.10716 [math.AT]},
  year={2024}
}

@article{huang_theriault_blowups,
  title={Homotopy of blow ups after looping},
  author={Huang, R. and Theriault, S.},
  journal={arXiv:2308.14531 [math.AT]},
  year={2023}
}

@article{huang_theriault_stability, 
   title={Stabilization of {P}oincar\'{e} Duality complexes and homotopy gyrations}, 
   author={Huang, R. and Theriault, S.}, 
   journal={to appear in Journal of the London Mathematical Society, arXiv:2504.09786 [math.AT]}, 
   year={2025} 
}

@article{ChenTher:gy_stab,
  title={Gyration Stability for Projective Planes},
  author={Chenery, S. and Theriault, S.},
  journal={Topology and its Applications},
  volume={369},
  pages={Paper No. 109420},
  year={2025},
  publisher={Elsevier}
}

@article{stanton_theriault,
  title={Loop spaces of $n$-dimensional {P}oincar{\'e} duality complexes whose $(n-1)$-skeleton is a co-{$H$}-space},
  author={Stanton, L. and Theriault, S.},
  journal={Transactions of the American Mathematical Society}, 
  volume={379}, 
  number={4}, 
  pages={2683-2715},
  year={2026}
}

@article{chenery:fico,
  title={On {F}ico's {L}emmata and the homotopy type of certain gyrations},
  author={Chenery, S.},
  journal={Bulletin of the London Mathematical Society},
  volume={58},
  number={1},
  pages={e70265},
  year={2026},
  publisher={Wiley Online Library}
}

@article{morava_SymplecticCobordism,
  title={Cobordism of Symplectic Manifolds and Asymptotic Expansions},
  author={Morava, J.J.},
  journal={Proceedings of the Steklov Institute of Mathematics},
  volume={225},
  number={0},
  pages={276--283},
  year={1999},
  publisher={Russian Academy of Sciences, Steklov Mathematical Institute}
}

@article{HJSX,
  title={Flat connections and the commutator map for {$\mathrm{SU}(2)$}},
  author={Ho, N.-K. and Jeffrey, L. and Selick, P. and Xia, E. Z.},
  journal={The Quarterly Journal of Mathematics},
  volume={72},
  number={1-2},
  pages={163--197},
  year={2021},
  publisher={Oxford University Press UK}
}

@article{atiyah-bott,
  title={The {Y}ang-{M}ills equations over {R}iemann surfaces},
  author={Atiyah, M. F. and Bott, R.},
  journal={Philosophical Transactions of the Royal Society of London. Series A, Mathematical and Physical Sciences},
  volume={308},
  number={1505},
  pages={523--615},
  year={1983},
  publisher={The Royal Society London}
}

\end{document}